\documentclass[]{amsart}
\usepackage{amssymb, mathdots}
\usepackage{mathrsfs}
\usepackage[utf8x]{inputenc}
\usepackage{amsmath, graphicx}
\usepackage{wrapfig}
\usepackage{diagrams}
\usepackage{enumerate}
\usepackage{url}
\DeclareSymbolFont{bbold}{U}{bbold}{m}{n}
\DeclareSymbolFontAlphabet{\mathbbold}{bbold}
\def\qmod#1#2{{\hbox{}^{\displaystyle{#1}}}\!\big/\!\hbox{}_{
\displaystyle{#2}}}

\def\resto#1#2{{
#1\hskip 0.4ex\vline_{\hskip 0.2ex\raisebox{-0,2ex}
{{${\scriptstyle #2}$}}}}}

\def\restobig#1#2{{
#1\hskip 0.4ex\vline_{\hskip 0.2ex\raisebox{-0,2ex}
{{${ #2}$}}}}}

\def\A{{\mathbb A}}
\def\B{{\mathbb B}}
\def\C{{\mathbb C}}

\def\E{{\mathbb E}}

\def\G{{\mathbb G}}
\def\H{{\mathbb H}}

\def\L{{\mathbb L}}

\def\N{{\mathbb N}}

\def\P{{\mathbb P}}
\def\Q{{\mathbb Q}}
\def\R{{\mathbb R}}
\def\Z{{\mathbb Z}}

\def\textmap#1{\mathop{\vbox{\ialign{
                                  ##\crcr
      ${\scriptstyle\hfil\;\;#1\;\;\hfil}$\crcr
      \noalign{\kern 1pt\nointerlineskip}
      \rightarrowfill\crcr}}\;}}
\def\bigtextmap#1{\mathop{\vbox{\ialign{
                                  ##\crcr
      ${\hfil\;\;#1\;\;\hfil}$\crcr
      \noalign{\kern 1pt\nointerlineskip}
      \rightarrowfill\crcr}}\;}}
      
\newcommand{\cal}{\mathcal}
\def\textlmap#1{\mathop{\vbox{\ialign{
                                  ##\crcr
      ${\scriptstyle\hfil\;\;#1\;\;\hfil}$\crcr
      \noalign{\kern-1pt\nointerlineskip}
      \leftarrowfill\crcr}}\;}}
 
\def\ag{{\mathfrak a}}

\def\cg{{\mathfrak c}}

\def\g{{\mathfrak g}}

\def\kg{{\mathfrak k}}

\def\mg{{\mathfrak m}}
\def\ng{{\mathfrak n}}

\def\pg{{\mathfrak p}}
\def\qg{{\mathfrak q}}

\def\tg{{\mathfrak t}}
\def\ug{{\mathfrak u}}

\def\Ag{{\mathfrak A}}

\def\Pg{{\mathfrak P}}

\newtheorem{sz}{Satz}[section]
\newtheorem{thry}[sz]{Theorem}
\newtheorem{pr}[sz]{Proposition}
\newtheorem{re}[sz]{Remark}
\newtheorem{co}[sz]{Corollary}
\newtheorem{dt}[sz]{Definition}
\newtheorem{lm}[sz]{Lemma}

\def\tr{\mathrm {Tr}}
\def\End{\mathrm {End}}
\def\Aut{\mathrm {Aut}}

\def\U{\mathrm{U}}
\def\SU{\mathrm {SU}}

\def\PU{\mathrm {PU}}
\def\GL{\mathrm {GL}}

\def\u{\mathrm {u}}

\def\Pic{\mathrm {Pic}}

\def\NS{\mathrm{NS}}
\def\deg{\mathrm {deg}}
\def\Hom{\mathrm{Hom}}

\def\Tors{\mathrm{Tors}}
\def\Vol{\mathrm{Vol}}
\def\vol{\mathrm{vol}}

\def\id{ \mathrm{id}}

\def\rk{\mathrm {rk}}
\def\ad{\mathrm {ad}}

\def\U2{\mathrm{U(2)}}
\def\U{\mathrm{U}}
\def\niq{=\kern-.18cm /\kern.08cm}
\def\Ad{\mathrm {Ad}}

\def\ad{{\rm ad}}

\def\ch{\mathrm{ch}}
\def\td{\mathrm{td}}
\def\HE{\mathrm{HE}}
\def\st{\mathrm{st}}
\def\ASD{\mathrm{ASD}}
\def\SL{\mathrm{SL}}

\begin{document}

\title[Kähler classes on universal moduli spaces and volumina]{Kähler classes on universal moduli spaces and volumina of Quot spaces}
\author{Christian Okonek \and  Andrei Teleman}

\begin{abstract}
We study the natural Kähler metrics on moduli spaces of stable oriented pairs in a very general framework, and we prove a universal formula expressing the Kähler class of such a moduli space  in terms of  characteristic classes of the universal bundle. We use these results to compute explicitly the volumina of certain Quot spaces.
\end{abstract}

\maketitle

\tableofcontents 

\section{Introduction}\label{Intro}

Let $(X,g)$ be compact Hermitian manifold and $E$ a ${\cal C}^\infty$ Hermitian  bundle on $X$. Recall that a Hermitian connection   $A$ on $E$ is called  Hermite-Einstein if its curvature $F_A$ is of type (1,1) and satisfies the Hermite-Einstein equation
\begin{equation}\label{HE}
i\Lambda_g F_A=c\id_E \ ,\
\end{equation}
for a constant $c\in\R$.
 
The classical version of the Kobayashi-Hitchin correspondence states that a holomorphic bundle ${\cal E}$ over a compact Gauduchon manifold $(X,g)$ is stable if and only if it admits a Hermitian metric $h$ such that the associated Chern connection $A_{{\cal E},h}$ is Hermite-Einstein and irreducible. This statement has first  been proved by Donaldson \cite{Do1} for projective-algebraic surfaces and for projective manifolds endowed with Hodge metrics \cite{Do2}, later  by Uhlenbeck and Yau for Kähler manifolds \cite{UY1}, \cite{UY2}, and finally by Li und Yau for arbitrary compact complex manifolds endowed with Gauduchon metrics (\cite{LY}, \cite{LT1}).

This statement, which gives a differential geometric characterization of stability, can be interpreted as an isomorphism between moduli spaces; more precisely, it yields a real analytic isomorphism
$${\cal M}^\HE(E)^*\textmap{KH} {\cal M}^\st(E)\ ,
$$ 
where   ${\cal M}^\HE(E)^*$ is the moduli space of irreducible  Hermite-Einstein connections on $E$, and ${\cal M}^\st(E)$ denotes the moduli space of stable holomorphic structures on $E$. This geometric interpretation of the original Kobayashi-Hitchin correspondence found important applications in differential topology: it allowed Donaldson to describe explicitly certain  instanton moduli spaces  on algebraic surfaces and to compute the first Donaldson invariants. A second fundamental application of the classical Kobayashi-Hitchin correspondence was Donaldson's non-vanishing theorem for his polynomial invariants on algebraic surfaces. This result was based on the fundamental fact that the Kähler class of the canonical metric of the moduli space can be identified with a tautological cohomology class \cite{DK}.  One of the motivations of this article is to show that this phenomenon holds in great generality.

Donaldson theory for $\PU(2)$-connections leads naturally to moduli spaces of projectively anti-seldual unitary connections  with fixed determinant  on Riemannian 4-manifolds (see \cite{Do3}).   More precisely, let $E$ be a  Hermitian vector bundle of rank 2 on  a Riemannian 4-manifold $(M,g)$,  $a\in{\cal A}(\det(E))$ a fixed Hermitian connection.     Denote by ${\cal A}_a^\ASD(E)$ the space of connections $A$ on $E$ inducing $a$ on $\det(E)$ which  satisfy the projective anti-selduality equation
\begin{equation}\label{HE0}
(F_A^0)^+=0\ .
\end{equation}
Here $F_A^0$ is the trace-free part of the curvature $F_A$, and $(\cdot)^+$ denotes the projection onto the self-dual summand $iA^2_+$. The moduli space which is relevant in Donaldson theory is the quotient
$${\cal M}^\ASD_a(E):=\qmod{{\cal A}_a^\ASD(E)}{\Gamma(M,\SU(E))}\ .
$$

On complex surfaces equation (\ref{HE0}) coincides with the projective Hermite-Einstein equation
$$i\Lambda_g F_A^0=0\ ,\ F_A^{0,2}=0\ ,
$$
which makes sense on arbitrary Hermitian complex manifolds and Hermitian bundles. In this case the fixed connection $a\in{\cal A}(\det(E))$ is chosen such that $F_a^{0,2}=0$, so that it defines a holomorphic structure ${\cal D}$ on the line bundle $\det(E)$. If the metric $g$ is Gauduchon one has again an isomorphism of moduli spaces
$${\cal M}^\HE_a(E)^*\textmap{KH_a} {\cal M}^\st_{\cal D}(E)\ ,
$$
where ${\cal M}^\HE_a(E)^*$ is the moduli space of irreducible projective  Hermite-Einstein connections  on $E$ inducing $a$ on $\det(E)$, modulo the real gauge group ${\cal G}:=\Gamma(X,\SU(E))$, and ${\cal M}^\st_{\cal D}(E)$ is the moduli space of stable holomorphic structures on $E$ inducing ${\cal D}$ on $\det(E)$, modulo the complex gauge group ${\cal G}^\C:=\Gamma(X,\SL(E))$.

The isomorphism $KH_a$ is an important refined version of the original Kobayashi-Hitchin correspondence $KH$. Indeed, on one hand one uses gauge groups associated with subgroups ($\SU(r)$, $\SL(r)$) of the structure groups, on the other hand one fixes some of the data ($a$, ${\cal D}$). The procedure of taking subgroups and fixing some of the data is called orientation \cite{LT2}.  

Another important generalization of the classical Kobayashi-Hitchin correspondence concerns moduli spaces of pairs consisting of (oriented) connections or holomorphic structures coupled with sections in associated bundles. For instance, in the theory of Hitchin pairs  \cite{Hi} one considers  a  differentiable Hermitian vector bundle $E$, a holomorphic structure  ${\cal D}$ on $\det(E)$, and a fixed  holomorphic bundle   ${\cal E}_0$  on $X$.  The complex geometric problem in this case is the classification of pairs $({\cal E},\varphi)$ consisting of a holomorphic structure ${\cal E}$ on $E$ inducing ${\cal D}$ on $\det(E)$, and a section $\varphi\in H^0({\cal E}nd({\cal E})\otimes {\cal E}_0)$, modulo the complex gauge group ${\cal G}^\C=\Gamma(X,\SL(E))$. In the classical case one takes ${\cal E}_0:=\Omega^1_X$  \cite{Hi}, \cite{Si}. 

Let $a$ be the Chern connection of the holomorphic structure ${\cal D}$ on the Hermitian line bundle $\det(E)$. The corresponding gauge theoretical problem concerns the moduli space of pairs $(A,\varphi)$ consisting of a Hermitian connection $A$ on $E$ inducing $a$ on $\det(E)$ and a section $\varphi\in A^0(\End(E)\otimes {\cal E}_0)$ such that the following Hermite-Einstein type equation is satisfied:

$$i\Lambda F_A^0+\frac{1}{2}[\Phi,\Phi^{*_h}]=0\ ,\ F_A^{0,2}=0$$
 Defining stability of Hitchin pairs in an appropriate way one obtains again a Kobayashi-Hitchin correspondence 
$${\cal M}^\HE_a(E,{\cal E}_0)\textmap{KH_{a,{\cal E}_0}} {\cal M}^\st_{\cal D}(E,{\cal E}_0) \ .
$$

Many other  important gauge theoretical problems fit into this concept, for instance:

\begin{enumerate}
 \item Moduli spaces of vortices \cite{Br}, \cite{GP}, \cite{HL}, \cite{Th}, 
\item Moduli spaces of oriented pairs \cite{OT1}, \cite{Te},
\item Moduli spaces of Witten triples \cite{W}, \cite{Bi}, \cite{Du}.

\end{enumerate}

All these Kobayashi-Hitchin type correspondences  can be interpreted, at least at a formal level, as infinite dimensional versions of the fundamental isomorphism between symplectic quotients (defined by a moment map) and GIT quotients (corresponding to a suitable stability condition) \cite{Ki}.

The idea of a universal Kobayashi-Hitchin correspondence, which specializes to all these isomorphisms of moduli spaces is therefore very natural and important.  General versions  have been obtained by Banfield \cite{Ban} for   connections in principal bundles over Kähler manifolds coupled with sections in associated vector bundles, and by Mundet i Riera \cite{Mu}  for connections in principal bundles over Kähler manifolds coupled with  sections in associated  Kählerian fibre bundles.   The universal Kobayashi-Hitchin correspondence of \cite{LT2} deals with (oriented) pairs   on Gauduchon manifolds, and identifies the complex geometric stability concept corresponding to the gauge theoretical equations. The final result is a universal isomorphism  $KH: {\cal M}^*\to {\cal M}^\st$  between a gauge theoretic moduli space ${\cal M}^*$ of irreducible (oriented) Hermite-Einstein pairs  and a complex geometric moduli space ${\cal M}^\st$ of stable (oriented) pairs.

 In this very general framework it is shown that the moduli space ${\cal M}^*$   comes with a natural metric, which is Hermitian with respect to the complex structure induced by the  Kobayashi-Hitchin isomorphism $KH$ and   is strongly  KT, i.e., its Kähler form $\Omega$ satisfies the equation $\partial\bar\partial \Omega=0$.
The idea of the proof  is to show that   $\Omega$ can be written as a sum  $\Omega=\Omega_1+\Omega_2$  of two terms, each term $\Omega_i$ being  obtained by fibre integration. More precisely
 $$\Omega_1=p_{{\cal M}^*}(\alpha_1\wedge p_X^*(\omega_g^{n-1}))\ ,\ \Omega_2=(n-1)!p_{{\cal M}^*}(\alpha_2\wedge p_X^*(\vol_g))\ ,$$
  where $\alpha_1$ and $\alpha_2$  are {\it closed} forms on ${\cal M}^*\times X$. Note that $\alpha_1$ is a characteristic form of the universal bundle $\P$ on ${\cal M}^*\times X$, hence the first term $\Omega_1$ is constructed in the same way as the de Rham representatives of the  Donaldson tautological classes \cite{DK}. 
 The first goal  of this article is to prove  that, under very general assumptions, the  closed form $\alpha_2$ is also a characteristic form  of $\P$ and to identify the corresponding  characteristic class explicitly. This is a highly non-trivial result whose proof is based on a push-forward formula in equivariant cohomology.   A similar, but less general result, has been obtained with different methods by Baptista \cite{Bap}.
 
 Therefore, assuming that $d(\omega_g^{n-1})=0$ (i.e.,   $g$  is semi-Kähler)  one obtains a canonical Kähler metric on ${\cal M}^*$ whose Kähler form is the  de Rham representative of  the sum of two Donaldson type classes. This result  can be regarded as a universal  generalization of formulae obtained by Manton–Nasir \cite{MN} and Perutz \cite{P} in the  case of vortices on Riemann surfaces.  
\vspace{4mm}

In order to formulate our main results, we recall briefly the formalism introduced in \cite{LT2}  in the special case of Kähler manifolds, which  suffices for our purposes.\\

Suppose that $(X,J,g)$ is an  $n$-dimensional  Kähler manifold and $\pi: P\to X$ is a   principal $  K$-bundle on $X$, with $ K$  a compact Lie group. We fix an epimorphism $r:  K\to  K_0$ of compact Lie groups, and denote by $N$  its kernel. Let $P_0:=  P\times_K   K_0$ be the associated $K_0$-bundle and fix a connection $A_0\in{\cal A}( P_0)$ of type $(1,1)$. 

Suppose we are given a left  $K$-action on a Kähler manifold $(F,J_F,g_F)$  by holomorphic isometries,  and  a moment map $\mu_F$ for this action on the corresponding symplectic manifold $(F,\omega_F)$.  We fix a $K$-invariant inner product $k$ on $\kg$ and we denote by $m_F:F\to  \kg$ the $ \kg$-valued map defined by $\mu_F=k(m_F,\cdot)$.  Let $p_\ng:\kg\to\ng$ be the $k$-orthogonal  projection onto the Lie algebra $\ng$ of $N$.  
The symmetric bilinear form  $h:\kg\times\kg\to\R$ defined by
$$h(x,y):=k(p_\ng (x),p_\ng(y))
$$
is $K$-invariant and non-degenerate on $\ng$.

Let   $E:=  P\times_{K} F$ be the associated bundle. The gauge group  
$${\cal G}:=\Gamma (X,  P\times_{K} N)$$
  acts  from the left on the space of $A_0$-oriented pairs ${\cal A}_{A_0}(  P)\times \Gamma(X,E)$. This space has a natural Kähler metric depending on the triple $(k,g,g_F)$ whose Kähler form is given by the formula (\cite{LT2} p. 66):
  $$\Omega((\alpha,\psi),(\beta,\chi):=\int_X h(\alpha\wedge\beta)\wedge\omega_g^{n-1}+(n-1)!\int_X \omega_F(\psi,\chi)\vol_g\ 
  $$

The fundamental object associated with these data is the moduli space  ${\cal M}^*$ of equivalence classes of irreducible oriented pairs $(A,\varphi)\in{\cal A}_{A_0}(P)\times \Gamma(X,E)$ satisfying the generalized  vortex  equations
$$\left\{
\begin{array}{ccc}
F_A^{02}&=&0\ \ \ \\\ 
\varphi  \hbox{ is }  A\hbox{-holomorphic }\\
p_{\ng}\left[  \Lambda_g  F_A+ m_F(\varphi)\right] &=&0\ .
\end{array}
\right. \eqno{(V)}
$$
In the third equation we used the same symbol   for the vector bundle epimorphism   $P\times_{ K}  \kg\to P\times_{  K}\ng$ induced by $p_\ng:\kg\to\ng$.

Denote by  $\omega_{\cal M^*}$ the Kähler form of this canonical Kähler metric on the moduli space  ${\cal M}^*$ of irreducible $A_0$-oriented pairs. Using a standard construction one defines a principal $K$-bundle $\P$ over ${\cal M}^*\times X$ and a tautological $K$-equivariant map $\Phi:\P\to F$, which can be regarded as a universal section in the associated $F$-bundle. Our main results are:

\newtheorem*{th-KformGen}{Theorem \ref{KformGen}} 
\begin{th-KformGen}
 The canonical Hermitian metric on ${\cal M}^*$  is Kähler and its Kähler form is
$$\omega_{\cal M^*}=p_{{\cal M}^*,*}\left[-\frac{1}{2}[ h(\Omega_\A\wedge\Omega_\A)]\wedge p_X^*(\omega_g^{n-1})+(n-1)!\eta\wedge p_X^*(\vol_g)\right]\ ,
$$
where $\Omega_{\A}$ denotes the curvature of the universal connection, and $\eta$ is the horizontal part of $ \Phi^*(\omega_F)-\langle \mu_F\circ\Phi,\Omega_\A\rangle$; $\eta$ is a closed 2-form. If $\mu_F$ can be written as $\mu_F=\mu_0+\tau$, where $\mu_0$ is an exact moment map and  $\tau\in   \kg^\vee$ is   $K$-invariant, then the de Rham cohomology class $[\omega_{\cal M^*}]_{\mathrm {DR}}$ is given by
$$[ \omega_{\cal M^*}]_{\mathrm {DR}}=p_{{\cal M}^*,*}\left[ -\frac{1}{2} [h(\Omega_\A\wedge \Omega_\A)] \cup  p_X^*[\omega_g]^{n-1}-(n-1)! [\tau(\Omega_\A)]\cup   p_X^*[\vol_g]\right] \ . 
$$
\end{th-KformGen}

 For the linear case we obtain a more precise result. If the chosen unitary representation
 $\rho: K\to \U(F)$ satisfies a technical condition (similar to a condition in \cite{Bap}), then one obtains an equality between forms (not just de Rham classes).  

\newtheorem*{th-Kform}{Theorem \ref{Kform}}
\begin{th-Kform}
Let  $F$ be a Hermitian vector space, and  $\rho: K\to \U(F)$ a unitary representation.  For a $ K$-invariant linear form $\tau\in  \kg^\vee$ put  $\mu_F:=\mu_0+\tau$, where $\mu_0$ is the standard moment map for the $  K$-action on $F$.

Suppose  there exists a $K$-invariant element $a_0\in \ng$ such that $\rho_*(a_0)=i \id_F$. Then the Kähler form $\omega_{\cal M^*}$ is given by

$$ \omega_{\cal M^*}=p_{{\cal M}^*,*}\left[ -\frac{1}{2}  h(\Omega_\A\wedge \Omega_\A)  \wedge  p_X^*(\omega_g)^{n-1}-(n-1)!  \tau(\Omega_\A) \wedge   p_X^*(\vol_g)\right] \ ,
$$ 
hence it  coincides with  fiber integration of a characteristic form  of the connection $\A$ on the universal bundle $\P$.  

\end{th-Kform}

The second goal of our article is to use these theoretical results to compute explicitly the volumina  of certain interesting moduli spaces with respect to the canonical metric induced by the Kobayashi-Hitchin correspondence.  


The computation of these volumina is not only an interesting mathematical problem, but is also important from a physical point of view. The first computation of such volumina is due to Manton and Nasir \cite{MN}. These authors  consider  the  vortex equation with fixed parameter $t=\frac{1}{2}$ on a   Riemann surface $(X,g)$ of genus $\g$. Assuming $\Vol_g(X)>4\pi d$ the symmetric power $X^{(d)}$ can be identified with the moduli spaces of vortices of degree $d$. Their main result    is   the computation of the volumina of   $X^{(d)}$ with respect to the induced metric (see \cite{MN}, (3.22)):
$$\Vol_{\omega}(X^{(d)})=\sum_{i=0}^{\min(d,g)} (4\pi)^i\left(\begin{matrix}\g\\ i\end{matrix}\right) \frac{1}{(d-i)!}(\Vol_g(X)-4\pi d)^{d-i}\ .
$$

A partial generalization of this result is due to Baptista, who computed the volumina of the moduli spaces of semi-local  Abelian vortices when the degree $d$ is larger than $2\g-2$ (see \cite{Bap} (52)).

Both formulae are (up to normalization factors) special cases of the formulae which we prove in this article (see section \ref{Applications}):\\

Let $X$ be a compact complex manifold of dimension $n$, ${\cal E}_0$ a locally free sheaf of rank $r_0$ on $X$,  and $E$ a differentiable vector bundle on $X$. We denote by ${\cal Q}uot^E_{{\cal E}_0}$ the Quot space of equivalence classes of quotients $q:{\cal E}_0\to {\cal Q}$ with locally free kernel of differentiable type $E$. When $E$ is of rank 1 and $m:=c_1(E)\in\NS(X)$ we put ${\cal Q}uot^m_{{\cal E}_0}={\cal Q}uot^E_{{\cal E}_0}$.

We fix a Hermitian metric $h_0$ on ${\cal E}_0$ and  denote by $A_0$ the associated Chern connection.
Let $g$  be a Kähler metric on $X$, $t\in\R$ a parameter and put $\tg:=\frac{(n-1)!\Vol_g(X)}{2\pi}t$. The moduli space ${\cal M}_t^*(E,A_0)$ of irreducible $t$-vortices of type $(E,A_0)$ can be identified via the Kobayashi-Hitchin correspondence with the moduli space ${\cal M}^\st_\tg(E,{\cal E}_0)$ of $\tg$-stable pairs of type $(E,{\cal E}_0)$, and for sufficiently large $\tg$ the latter can be identified with the Quot space ${\cal Q}uot^E_{{\cal E}_0}$ \cite{OT2}. In this way we obtain   natural metrics on  ${\cal Q}uot^E_{{\cal E}_0}$.

Note that our moduli spaces come with a natural symmetry, since any  morphism of complex Lie groupes $G\to\Aut({\cal E}_0)$ induces a natural holomorphic $G$-action on the moduli spaces  ${\cal M}^\st_\tg(E,{\cal E}_0)$ and ${\cal Q}uot^E_{{\cal E}_0}$.    Our  formalism yields  a natural $G$-equivariant lift 
$$[\omega_t]^G\in H^2_G({\cal M}^\st_\tg(E,{\cal E}_0),\R)$$
 of the Kähler class $[\omega_t]$ in terms of equivariant Chern classes of the universal bundle.%
\newtheorem*{equivariant-omega-intro}{Theorem \ref{Symm}}
\begin{equivariant-omega-intro}

Let $\mathscr{E}$ be the universal bundle on ${\cal M}^\st_{\tg} (E,{\cal E}_0)\times X$. For any morphism of Lie  groups $G\to \Aut({\cal E}_0)$, the class
$$ \left\{-4\pi^2\mathrm{ch}_2^G(\mathscr{E})  \cup  p_X^*[\omega_g^{n-1}]-2t\pi(n-1)!  c_1^G(\mathscr{E}) \cup   p_X^*[\vol_g]\right\}/[X] 
$$
is a lift of  the Kähler class $[\omega_{t}]$ to $H^2_G({\cal M}^\st_{\tg} (E,{\cal E}_0),\R)$.

\end{equivariant-omega-intro}

This general result is very useful  for explicit computations of  volumina  of moduli spaces using   localization methods.
\\

In the case when $E$ is a line bundle with $c_1(E)=m$ we have a natural  identification  ${\cal M}^\st_\tg(E,{\cal E}_0)= {\cal Q}uot^E_{{\cal E}_0}$ for  any $\tg> -\deg(E)$, and
\begin{equation}\label{intro-Perutz-new}\frac{1}{4\pi^2}[\omega_{{\cal M}^*_t}]=\theta+(\deg(E)+\tg)\gamma\ ,
\end{equation}
\def\st{\mathrm{st}}
where $\theta$ is the pull-pack of the theta class of $\Pic^{m}(X)$, and $\gamma=-c_1(\resto{\mathscr{E}}{{\cal Q}uot^E_{{\cal E}_0}\times\{x_0\}})$.

Formula (\ref{intro-Perutz-new}) specializes to a result obtained by Perutz for $r_0=1$ and $n=1$ (see \cite{P} Theorem 3 ).

 When $X$ is a Riemann surface and $r_0=1$, the Quot space ${\cal Q}uot^m_{{\cal E}_0}$ can be identified with the symmetric power $X^{(d)}$, $d:=\deg({\cal E}_0)-\deg(E)$. Using known identities for the intersection numbers $\langle\theta^i\gamma^j,[X^{(d)}]\rangle$ we obtain 
$$\Vol_{\omega_t}({\cal Q}uot^E_{{\cal E}_0})=(4\pi^2)^d\sum_{i=1}^{\min(d,\g)}\left(\begin{matrix} \g\\ i\end{matrix}\right)\frac{1}{(d-i)!}(\deg(E)+\tg)^{d-i}\ ,
$$
which (up to the normalization factor $\pi^2$) specializes to the Manton-Nasir formula.

Our second application concerns the volume of the Abelian Quot spaces ${\cal Q}uot^E_{{\cal E}_0}$ for arbitrary $r_0\geq 1$, $n\geq 1$. We will say that a pair $(m,{\cal E}_0)$ acyclic if $H^i({\cal L}^\vee\otimes{\cal E}_0)=0$ for all $i>0$ and all $[{\cal L}]\in\Pic^m(X)$. If this is the case then  ${\cal Q}uot^E_{{\cal E}_0}$ is a projective bundle over $\Pic^m(X)$, and we obtain an explicit formula for its volume:
\begin{equation}\label{VolAcyc} \Vol_{\omega_t}({\cal Q}uot^E_{{\cal E}_0})=\frac{(4\pi^2)^N}{N!} \bigg\{\bigg(\sum_{k=R-1}^N  \left(\begin{matrix}N\\ k\end{matrix}\right)  (\deg_g(E)+\tg)^{N-k}    \theta^k \bigg)\wedge \ \ \ \ \ \ \ \ \ \ \ \ \ \ \ \ \ \ \ \ \ \ \ $$
$$\ \ \ \ \ \ \ \ \ \ \ \ \ \ \ \ \ \ \ \ \ \ \ \ \ \ \ \ \ \ \wedge\exp\bigg({\sum_{i=1}^q \sum_{s=0}^{n-i} \frac{(-1)^s}{i s!} \kg_{m^s C_{n-i-s}}  }\bigg)\bigg\}(h_1,\dots,h_{2q}) \ .\end{equation}
In this formula $(h_1,\dots,h_{2q})$ is a basis of $H^1(X,\Z)$ compatible with its natural orientation,  $R:=\chi({\cal L}^\vee\otimes{\cal E}_0)$, and $N:=R+q(X)-1$ is the dimension of ${\cal Q}uot^E_{{\cal E}_0}$. The classes $C_i\in H^{2i}(X,\Q)$ are defined by the identity $\ch({\cal E}_0)\td(X)=\sum_i C_i$, and  for a class $\cg\in H^{2n-k}(X,\Q)$ we denote  by $\kappa_\cg\in\mathrm{Alt}^k(H^1(X,\Z),\Q)$  the form defined by
$$\kg_\cg(x_1,\dots,x_k):=\langle x_1\cup\dots \cup x_k\cup\cg, [X]\rangle\ .
$$

In the special case when $n=1$ and ${\cal E}_0={\cal O}_X^{\oplus r_0}$, formula (\ref{VolAcyc}) specializes to Baptista's result mentioned above.

The last section concerns the computation of the volumina of non-Abelian  Quot spaces ${\cal Q}uot^E_{{\cal E}_0}$ on Riemann surfaces in the case $r=r_0$. These moduli spaces are always smooth and projective, but their cohomology algebra is unknown. Since for ${\cal E}_0={\cal O}_X^{\oplus r_0}$ one obtains precisely Weil's moduli spaces of matrix divisors,  we call them  moduli spaces of {\it twisted matrix divisors}. For   fixed data $(E,t,g)$ the volume of ${\cal Q}uot^E_{{\cal E}_0}$ depends only on the topological type of ${\cal E}_0$, hence we can suppose that ${\cal E}_0=\bigoplus_{i=1}^{r_0} {\cal L}_i$ splits  as a direct sum of holomorphic line bundles  ${\cal L}_i$. In this case one obtains a natural action of $[\C^*]^{r_0}$ on ${\cal Q}uot^E_{{\cal E}_0}$ and the computation of the volume of this moduli space can be reduced  -- using localization techniques -- to an integration over the fixed point locus of this action. 
For our computations we will use the $\C^*$-action on ${\cal Q}uot^E_{{\cal E}_0}$ induced by the morphism $\C^*\to [\C^*]^{r_0}$ associated with a system $(w_1,\dots,w_{r})$ of pairwise distinct weights. With this choice we obtain
$$[{\cal Q}uot^E_{{\cal E}_0}]^{\C^*}=\coprod_{\underline{d}\in I(d)}\prod_{i=1}^{r}{\cal Q}uot^{l_i-d_i}_{{\cal L}_i}\ ,
$$
where $d:=\deg({\cal E}_0)-\deg({\cal E})$,   $l_i:=\deg({\cal L}_i)$ and
$$I(d):=\{\underline{d}=(d_1,\dots,d_r)\in\N^r|\ \sum_{i=1}^r d_i=d\}\ .
$$
  Using Theorem \ref{Symm} we obtain an explicit expression for the   the equivariant restrictions
$$\restobig{[\omega_t]^{\C^*}}{\prod_{i=1}^{r}{\cal Q}uot^{l_i-d_i}_{{\cal L}_i}}\ .
$$
In order to apply the localization formula we also need the equivariant Euler class of the normal bundles of the components $\prod_{i=1}^{r}{\cal Q}uot^{l_i-d_i}_{{\cal L}_i}$ of $[{\cal Q}uot^E_{{\cal E}_0}]^{\C^*}$ in ${\cal Q}uot^E_{{\cal E}_0}$. For this purpose   we follow closely similar computations of Marian and Oprea \cite{MO2}. 

We obtain a complicated   general formula for   $\Vol_{\omega_t}({\cal Q}uot^E_{{\cal E}_0})$, which is non-explicit because it implicitly involves the coefficients of the Taylor expansion of a transcendental function in $2r$ variables.  However, it  gives an algorithm which produces closed formulae for small $r$, $d$.   Deriving a general closed formula in terms of the topological data $(\g,d,l,r)$ and $\Vol_g(X)$ is an interesting  combinatorial problem.

Note that the computation of the volume of ${\cal Q}uot^E_{{\cal E}_0}$ gives a
  purely algebraic geometric result: The moduli spaces ${\cal Q}uot^E_{{\cal E}_0}$ of twisted matrix divisors come with natural Grothendieck embeddings
$$j_n:{\cal Q}uot^E_{{\cal E}_0}\hookrightarrow \P(V_n)
$$
in projective spaces (for $n\gg 0$). The Kähler class  $c_1(j_n^*({\cal O}_{V_n}(1)))$ can be identified with a multiple  of our Kähler class $[\omega_{t}]$ for a suitable choice of the parameter $t$, more precisely
$$c_1(j_n^*({\cal O}_{V_n}(1)))=\frac{1}{4\pi^2} \bigg[\omega_{\frac{2\pi(n-\g)}{\Vol_g(X)}} \bigg]\ .
$$

Therefore our computation of $\Vol_{\omega_t}({\cal Q}uot^E_{{\cal E}_0})$ yields in particular the degrees of the images of the Quot spaces ${\cal Q}uot^E_{{\cal E}_0}$ under the Grothendieck embeddings.

We illustrate all this with explicit formulae in the case $r=2$ and $d\in\{1,2\}$.

\section{General theory}

\subsection{Donaldson's quotient connection} \label{quotConn}

Suppose we have a $G$-equivariant principal $K$-bundle $\pi:P\to M$, where $G$ is a Lie group acting {\it freely} on $M$. By definition this means that $G$ acts on $P$ by $K$-bundle automorphisms covering the action on $M$.   Suppose that the quotient $\bar M:=M/G$ has a manifold structure such that the projection $q:M\to \bar M$ is a submersion. This implies that $q$ is a principal $G$-bundle over $\bar M$. 

In this situation  the quotient $\bar P:=P/G$ has a natural structure of a principal  $K$-bundle over $\bar M$. This bundle $\bar \pi:\bar P \to \bar M$ will be called the $G$-quotient of $\pi$.

Let $A$ be a $G$-invariant connection on $P$. The aim of this appendix is to construct a quotient connection $\bar A$ on $\bar P$ and to compute its curvature. The construction depends on the choice of an auxiliary     connection $\Gamma$   on the $G$-bundle $q:M\to \bar M$, i.e., a $G$-invariant  distribution on $M$ which is horizontal with respect to the projection $q:M\to \bar M$.\footnote{Note that $G$ acts on $M$ from the left  so, using the standard terminology of \cite{KN}, $\Gamma$ is a connection in the principal $G$-bundle obtained by endowing $M$ with the right action $ m\cdot g:= g^{-1} m$.}  

  The construction of this quotient connection follows the method explained in  \cite{DK} for linear connections, and later in the general framework of the {\it universal Kobayashi-Hitchin correspondence} in \cite{LT2}. We believe that  the general construction of the quotient connection and the computation of its curvature in the general framework of quotient of principal bundles is of independent interest, so we explain it briefly here. 

We denote by $\theta_A\in A^1(P,\kg)$ the connection form of $A$ and by  $\Gamma^A$ the horizontal distribution associated with $A$. The intersection $(\pi_*)^{-1}(\Gamma)\cap \Gamma^A$
defines  a  distribution on $P$ which is invariant with respect to both the $K$-action and the $G$-action, and is horizontal with respect to the projection $P\to \bar M$.  Let $\qg:P\to\bar P$ be the natural projection covering $q:M\to\bar M$.
\begin{dt} The quotient  connection $\bar A$ on the bundle $\bar\pi:\bar P \to 
\bar M$ is the connection associated with the distribution $\Gamma^{\bar A}:=\qg_*(\pi^{-1}_*(\Gamma)\cap \Gamma^A)$.
\end{dt}

The curvature of $\bar A$ can be computed as follows: Let $\bar \xi$, $\bar \eta$ be vector fields on $\bar M$ and let $\xi$, $\eta$ be their $\Gamma$-horizontal lifts on $M$. The Lie bracket $[\xi,\eta]$ can be decomposed as the sum $[\xi,\eta]^h+[\xi,\eta]^v$ of its $\Gamma$-horizontal and $\Gamma$-vertical components, and one has
$$[\xi,\eta]^v= \Omega_\Gamma(\xi,\eta)^\#_M\ .$$
 In this formula we use  the notation $a^\#_M$ for the   vector field on $M$ corresponding to an element $a\in\g$ via the given (left) action, whereas   $\Omega_\Gamma$ is the curvature of the connection defined by $\Gamma$ in the $G$-principal bundle $q:M\to \bar M$ constructed using the {\it right} action corresponding to the original left action on $M$. This explains why there is no $-$ sign in front of the right hand side.

Let $\tilde \xi$, $\tilde \eta$ be the $A$-horizontal lifts of $\xi$, $\eta$ to $P$. These vector fields are obviously sections of $\pi_*^{-1}(\Gamma)\cap \Gamma^A$. Let $m\in M$ and let $p\in P_m$ be a lift of $m$ in $P$. The $A$-horizontal and $A$-vertical  vertical components of the bracket $[\tilde \xi,\tilde \eta]$ in $p$ are:
\begin{equation}\label{dec}[\tilde \xi,\tilde \eta]^h_p= {[\xi,\eta]}^\sim_p=\{[\xi,\eta]^h_m\}^\sim_p+\{[\xi,\eta]^v_m\}^\sim_p=\{[\xi,\eta]^h_m\}^\sim_p+\{\Omega_\Gamma(\xi_m,\eta_m)^\#_{M,m}\}^\sim_p\ , $$
$$  
 [\tilde \xi,\tilde \eta]^v_p=-\Omega_A(\tilde \xi_p,\tilde \eta_p)^\#_{P,p}\ .
\end{equation}

On the other hand, since $\pi:P\to M$ is a morphism of left $G$-manifolds, it follows that for every $\alpha\in\g$ one has $\pi_*(\alpha^\#_{P,p})=\alpha^\#_{M,m} $, so that the $A$-horizontal component of $\alpha^\#_{P,p}$ is $\{\alpha^\#_{M,m}\}^\sim_p $. The $A$-vertical component of $\alpha^\#_{P,p}$ is $\{\theta_A(\alpha^\#_{P,p})\}^\#_{P,p}$  by  definition of the connection form $\theta_A$.   Therefore
$$\{\alpha^\#_{M,m}\}^\sim_p=\alpha^\#_{P,p}-\{\theta_A(\alpha^\#_{P,p})\}^\#_{P,p}
\ .
$$
Applying this formula to  $\alpha=\Omega_\Gamma(\xi_m,\eta_m)$ and using (\ref{dec}) we get
$$[\tilde \xi,\tilde \eta]_p=[\tilde \xi,\tilde \eta]^h_p+[\tilde \xi,\tilde \eta]^v_p $$
$$=\{[\xi,\eta]^h_m\}^\sim_p+\Omega_\Gamma(\xi_m,\eta_m)^\#_{P,p}-\{\theta_A(\Omega_\Gamma(\xi_m,\eta_m)^\#_{P,p})\}^\#_{P,p}-\Omega_A(\tilde \xi_p,\tilde \eta_p)^\#_{P,p}\ .
$$

Now we take the image of both sides of this equality  via the differential  of the projection $\qg:P\to \bar P$ at $p$. Put $\bar p:=\qg(p)$, and let $\tilde{\bar\xi}$, $\tilde{\bar\eta}$ be the $\bar A$-horizontal lifts of ${\bar\xi}$ and ${\bar\eta}$. One easily sees that the push-forwards $\qg_*(\tilde \xi)$ and $\qg_*(\tilde \eta)$ are defined and equal to $\tilde{\bar\xi}$ and  $\tilde{\bar\eta}$ respectively. From this one shows  that $\qg_*[\tilde \xi,\tilde \eta]=[\tilde{\bar\xi},\tilde{\bar\eta}]$. On the other hand   $\qg_*\left(\Omega_\Gamma(\xi_m,\eta_m)^\#_{P,p}\right)=0$, because the vector field $\Omega_\Gamma(\xi_m,\eta_m)^\#_P$ is tangent to the $G$-orbits. Finally   $\qg_*\left(\{[\xi,\eta]^h_m\}^\sim_p\right)$ is $\bar A$-horizontal, because $\{[\xi,\eta]^h_m\}^\sim_p$ belongs to $\pi^{-1}_*(\Gamma)\cap \Gamma^A$.  Therefore
$$[\tilde{\bar\xi},\tilde{\bar\eta}]_{\hat p}=\qg_*\left\{ \{[\xi,\eta]^h_m\}^\sim_p\right\}-\qg_*\left\{ \left\{ \theta_A(\Omega_\Gamma(\xi_m,\eta_m)^\#_{P,p}) +\Omega_A(\tilde \xi_p,\tilde \eta_p)\right\}^\#_{P,p}\right\}\ .
$$
Since $\qg$ is $K$-equivariant we have 
$$\qg_*\left\{ \left\{ \theta_A(\Omega_\Gamma(\xi_m,\eta_m)^\#_{P,p}) +\Omega_A(\tilde \xi_p,\tilde \eta_p)\right\}^\#_{P,p}\right\}= \left\{ \theta_A(\Omega_\Gamma(\xi_m,\eta_m)^\#_{P,p}) +\Omega_A(\tilde \xi_p,\tilde \eta_p)\right\}^\#_{\bar P,\bar p}\ .$$
But $\qg_*\left\{ \{[\xi,\eta]^h_m\}^\sim_p\right\}$ is $\bar A$-horizontal and  $\Omega_{\bar A}(\tilde{\bar \xi}_{\bar p},\tilde{\bar\eta}_{\bar p})^\#=-[\tilde{\bar\xi},\tilde{\bar\eta}]^v$. Hence we obtain
\begin{equation}\label{curvature}\left\{ \Omega_{\bar A}(\tilde{\bar\xi}_{{\bar p}},\tilde{\bar \eta}_{\bar p})\right\}^\#_{\bar P,\bar p}=\left\{\Omega_A(\tilde \xi_p,\tilde \eta_p)+\theta_A(\Omega_\Gamma(\xi_m,\eta_m)^\#_{P,p})\right\}^\#_{\bar P,\bar p}\ .
\end{equation}
This proves
\begin{pr}\label{CurvProp} Let $m\in M$, $p\in P_m$, $\bar m:=q(m)$, $\bar \xi$, $\bar \eta\in T_{\bar m}\bar M$, $\xi$, $\eta\in T_m M$ their $\Gamma$-horizontal lifts, $\tilde \xi$, $\tilde  \eta\in T_pP$ the  $A$-horizontal lifts of $\xi$, $\eta$, and $\tilde {\bar \xi}$, $\tilde{\bar \eta}\in T_{\bar p}\bar P$  the $\bar A$-horizontal lifts of $ \bar \xi$, $\bar  \eta$ at the point $\bar p:=\qg(p)$. The curvature of the quotient connection $\bar A$ is given by the formula
\begin{equation}\label{curv-formula} \Omega_{\bar A}(\tilde{\bar\xi},\tilde{\bar \eta} )=\Omega_A(\tilde \xi ,\tilde \eta )+\theta_A(\Omega_\Gamma(\xi,\eta)^\#_{P,p})\ .
\end{equation}
\end{pr}

\subsection{A pushforward formula in equivariant   cohomology}\label{Cartan}

Let $K$ be a compact Lie group, and $F$ a differentiable manifold endowed with a left $K$-action. The Cartan algebra of the $K$-manifold $F$ (see \cite{GGK} p. 198) is  the differential graded  $\R$-algebra $A_K^*(F):=\big[\R[\kg]\otimes A^*(F)\big]^K$ of $K$-invariant $A^*(F)$-valued polynomials on the Lie algebra $\kg$;  the grading given by $\deg(f\otimes \eta):=2\deg(f)+\deg(\eta)$, and the differential is 
\begin{equation}\label{dK-left}d_K(\alpha)(a):=d(\alpha(a))+\iota_{a^\#}(\alpha(a))\  ,   
\end{equation}
for any $\alpha\in \R[\kg]\otimes A^*(F)$ regarded as a polynomial map $\kg\to A^*(F)$.  Here $a\in\kg$ and $a^\#$ stands for the vector field on $F$ associated with $a$. The cohomology $H_K^*(F)$ of this algebra is canonically isomorphic to the real equivariant cohomology $H^*_K(F,\R)$.  A  differentiable equivariant  map $F\to F'$ defines a morphism of differential graded  $\R$-algebras $\varphi^*_K:A^*_K(F')\to A^*_K(F)$ (given by pullback of forms), and the assignment $F\to A_K^*(F)$   defines a contravariant functor on the category of $K$-manifolds with values in the category of differential graded  $\R$-algebras.

In order to extend the functoriality of the assignment $F\to A_K^*(F)$  to right $K$-manifolds and  equivariant maps between a {\it right} and a left $K$-manifold,  one introduces  the Cartan algebra associated with a right $K$-manifold $P$ by endowing the  graded $\R$-algebra $A_K^*(P):=\big[\R[\kg]\otimes A^*(P)\big]^K$ with the differential 
\begin{equation}\label{dK-right}d_K(\alpha)(a):=d(\alpha(a))-\iota_{a^\#}(\alpha(a))\ .    
\end{equation}

With this definition, any differentiable equivariant map $\varphi:P\to F$ from a right $K$-manifold to a left $K$-manifold (i.e. any differentiable map $\varphi:P\to F$ satisfying the identity $\varphi(pk)=k^{-1}\varphi(p)$)  induces a morphism $\varphi^*_K:A^*_K(F)\to A^*_K(P)$ of {\it differential} graded  $\R$-algebras defined by pull-back of forms. In order to see this one uses the identity 
\begin{equation}\label{sign-change}
\varphi_{*}(a^\#_p)=-a^\#_{\varphi(p)}\ \forall p\in P\ \forall a\in\kg\  
\end{equation}
to show that for such a map $\varphi$ one has $\varphi^*\circ \iota_{a^\#}=-\iota_{a^\#}\circ \varphi^*$ as maps from $ A^*(F)$ to  $A^*(P)$.

 Let now $\pi:P\to B$ be a principal $K$-bundle on $B$.  We have the standard $\R$-algebra isomorphism
$$H_K(\pi):H^*(B,\R)\textmap{\simeq} H^*_K(P,\R)\  .$$
Using the de Rham   algebra of $B$ and the Cartan   algebra of the right $K$-manifold $P$,  the isomorphism $H_K(\pi)$ corresponds to the map $H_K(\pi):H^*_{\rm DR}(B)\to H^*_K(P)$ defined by
$$H_K(\pi)([\eta]_\mathrm{DR})=[\pi^*(\eta)]_\mathrm{C}\ .
$$
Here  $\pi^*(\eta)$ is regarded as a constant, $K$-invariant $A^*(P)$-valued polynomial map on $\kg$, and $[\cdot ]_\mathrm{C}$ denotes the Cartan cohomology class.

  We will show that, choosing a connection $A$ on the principal bundle $P$, one can give an explicit formula for the inverse map $H_K(\pi)^{-1}:  H^*_K(P)\to H^*_\mathrm{DR}(B)$.  

Fix a connection $A$ on $P$. We recall that a $k$-form $\beta$ on $P$ is called horizontal if $\iota_X\beta=0$  for any vertical tangent vector field $X$, and is called basic if is horizontal and $K$-invariant. The algebra $A^*(P)_{\rm ba}$ of basic forms on $P$ is $d$-invariant, and the differential graded  algebra $(A^*(P)_{\rm ba},d)$  can be identified with the differential graded  algebra $(A^*(B),d)$ in the obvious way.

 Denote by $(\cdot)^h$ the projection $T_P\to \Gamma_A$ of the tangent space of $P$ on the horizontal distribution $\Gamma_A\subset T_P$ defined by $A$. We will use the same symbol for the linear operator which maps a differential form on $P$  to the corresponding horizontal form, i.e., 
$$\beta^h(X_1,\dots, X_k):=\beta(X_1^h,\dots,X_k^h)\ .
$$

\begin{pr}\label{inverse} Let $A$ be a connection on $P$ and $\Omega_A\in A^2(P,\kg)$ its curvature form. For any $\alpha\in A^*_K(P)$ the form $\alpha(\Omega_A)^h\in A^*(P)$ is basic,  the map $\alpha\mapsto \alpha(\Omega_A)^h$ defines a morphism of differential graded  $\R$-algebras 
$$\pi_*^A:A^*_K(P)\to A^*(P)_{\rm ba}=A^*(B)\ ,$$
 and the induced map $H(\pi_*^A):H^*_K(P)\to  H^*_{\mathrm DR}(B)$ coincides with $H_K(\pi)^{-1}$. 
\end{pr}

This statement is similar to Proposition 7.34 in \cite{BGV}\footnote{Note however that, with conventions introduced in section 7.1 of \cite{BGV}, there is a sign error in Proposition 7.34 and its proof. The proof uses the same formula (\ref{dK-right})   for  both left and right actions, and does not take into account the sign  change (\ref{sign-change}). }. For completeness we include a short proof based on the same method. 

\begin{proof} The fact that $\alpha(\Omega_A)^h$ is basic follows from the invariance properties of $\alpha$ and $\Omega_A$. The differential $d\big(\alpha(\Omega_A)^h\big)$ is automatically basic, so it suffices to compute it on horizontal vector fields, in other words is sufficient to compute the covariant derivative $D_A\big(\alpha(\Omega_A)\big)$, where $D_A:A^*(P)\to A^*(P)$ is defined by  $D_A(\sigma)=(d(\sigma^h))^h$. Fixing a basis $(a_i)_i$ in $\kg$, putting $\Omega_A=\Omega_{A}^i a_i$, and using Lemma 7.31 in \cite{BGV} we obtain  for a tensor product $  f\otimes\beta\in \R[\kg]\otimes A^*(P)$
$$D_A\big((f\otimes \beta)(\Omega_A))= D_A\big(  f (\Omega_A)\wedge\beta)=\left\{d(f(\Omega_A)\wedge\beta) -\Omega_A^i\ \iota_{a_i^\#}(f(\Omega_A)\wedge\beta)\right\}^h\ .
$$

Since $(d\Omega_A)^h=0$ by Bianchi's identity and $\iota_{a_i^\#}\Omega_A=0$ (because $\Omega_A$ is a horizontal form), we obtain

$$D_A\big((f\otimes \beta)(\Omega_A))= \left\{ (f(\Omega_A)\wedge d\beta)-\Omega_A^i\ (f(\Omega_A)\wedge\iota_{a_i^\#}\beta)\right\}^h$$
$$=\left\{ f(\Omega_A)\wedge (d\beta -  \Omega_A^i\wedge\iota_{a_i^\#}\beta)\right\}^h=\left\{d_K(f\otimes \beta)(\Omega_A)\right\}^h\ .
$$
Therefore, for every $\alpha  \in \R[\kg]\otimes A^*(P)$ ($K$-invariant or not) one has
$$D_A(\alpha(\Omega_A))=\{(d_K(\alpha))(\Omega_A)\}^h\ .$$
This implies 
$$(d\circ \pi^A_*)(\alpha)=(\pi^A_*\circ d_K)(\alpha) 
$$
for  every $\alpha\in \big[\R[\kg]\otimes A^*(P)\big]^K$.
The fact that the induced map $H(\pi_*^A):H^*_K(P)\to  H^*_{\mathrm{DR}}(B)$ coincides with $H_K(\pi)^{-1}$ is now obvious, since for any form $\eta\in A^*(B)$, the basic form $\pi^A_*(\pi^*(\eta))$ coincides with $\pi^*(\eta)$, so the corresponding form in $A^*(B)$ is $\eta$.
\end{proof}
\begin{re} The Chern-Weil  morphism $\R[\kg]^K\to Z^*_{\rm DR}(B)$ can be obtained as the composition $ \pi_*^A \circ j_K^P$, where $j_K^P:\R[\kg]^K\hookrightarrow A^*_K(P)$ is the obvious embedding. For an invariant polynomial $f\in \R[\kg]^K$, the cohomology class $[\pi_*^A (f)]_{\rm DR}$ is the characteristic class  $c_f$ associated with $f$. 
\end{re}

\subsection{Sections in Hamiltonian fibre bundles}

Let $\pi:P\to B$ be a principal $K$-bundle on $B$, $F$ a manifold endowed with left $K$-action, and let $E:=P\times_K F$ be the associated bundle with fiber $F$.  In this section we study the morphism $H_K^*(F)\to H^*_{\rm DR}(B)$ defined by a section $\varphi\in\Gamma(B,P\times_K F)$.

It is   well known that the data of a section $\varphi$ in $E$ is equivalent to the data of a $K$-equivariant map $P\to F$. We will denote the two objects by the same symbol.  Using the formalism and the results  of section \ref{Cartan}, we obtain a morphism $H_K(\varphi^*_K):H^*_K(F)\to H^*_K(P)$, hence a morphism
$$H_K(\pi)^{-1}\circ H_K(\varphi^*_K): H^*_K(F)\to H_{\rm DR}^*(B)\ .  
$$
This morphism coincides with $H(\pi_*^A)\circ H_K(\varphi^*_K)$,  and is given explicitly by the formula
$$[\alpha]_{\rm C}\mapsto \big[\{\varphi^*_K(\alpha)(\Omega_A)\}^h\big]   
$$
via the identification $A^*(B)=A^*(P)_{\rm ba}$. The   following result   was  stated in an equivalent way  in    \cite{LT2} section 6.2.2, where it was proved   without using equivariant cohomology  as a special case of the so-called the  ``generalized symplectic reduction".

\begin{pr} \label{omega-mu} Let $\omega$ be a $K$-invariant symplectic form on $F$, and $\mu:F\to \kg^\vee$ a moment map     for the $K$-action on the symplectic manifold $(F,\omega)$. Denote by the same symbol also the associated comoment map   $\kg\to A^0(F)$, which is an invariant $A^0(F)$-valued linear map on $\kg$. Then:
\begin{enumerate}
\item $d_K(\omega-\mu)=0$, hence $\omega-\mu$ defines a cohomology class $[\omega-\mu]_{\rm C}\in H^2_K(F)$.
\item The basic 2-form
$$
\left\{\varphi^*(\omega)-\langle \mu\circ\varphi,\Omega_A\rangle\right\}^h 
$$
is closed and represents  the de Rham class 
$$H_K(\pi)^{-1}\circ H_K(\varphi^*_K)([\omega-\mu]_{\rm C})\in H^2_{\rm DR}(B)\ .$$
 \end{enumerate}
\end{pr}
\begin{proof} One has $d_K(\omega-\mu)(a)=d((\omega-\mu)(a))+\iota_{a^{\#}}((\omega-\mu)(a))=-d\mu^a+\iota_{a^{\#}}(\omega)=0$. The second statement follows from Proposition \ref{inverse}.
\end{proof}

Suppose now that 
\begin{equation}\label{exact-omega}\omega=d\theta\ ,
\end{equation}
where $\theta$ is a $K$-invariant 1-form on $F$ \footnote{This happens for instance when $(F,\omega)$ is a Hermitian  vector space endowed with the canonical Kählerian  form, or when $F$ is the cotangent bundle of a $K$-manifold endowed with the canonical symplectic structure.}.
In this case one has a natural moment map $\mu_\theta$ for the $K$-action on  $(F,\omega)$ defined by
\begin{equation}\label{exact-mu}\langle \mu_\theta ,a \rangle:=-\iota_{a^\#}\theta\ \ \forall a\in\kg\ .
\end{equation}
Such a moment map is called exact (see \cite{GGK} Example 2.8 p. 18). 
Note that $\theta$ defines an element of $A^1_K(F)$ and  formulae (\ref{exact-omega}), (\ref{exact-mu}) are equivalent to
\begin{equation}\label{exact-omega-mu}
d_K(\theta)=\omega-\mu_\theta\ .
\end{equation}
Therefore the cohomology class $[\omega-\mu_\theta]_{\rm C}\in H^*_K(F)$  hence also its image in $H^2_{\rm DR}(B)$  vanishes. More precisely we have

\begin{co} \label{exact} With the notations and assumptions of Proposition \ref{omega-mu}, suppose that $\omega=d\theta$, for a $K$-invariant 1-form on $F$, and let $\mu_\theta$ be the corresponding  exact moment map. Then $\varphi^*(\theta)^h$ is a basic 1-form on $P$ and   
$$d(\varphi^*(\theta)^h)=\left\{\varphi^*(\omega)-\langle \mu_\theta\circ\varphi,\Omega_A\rangle\right\}^h\ .
$$
\end{co}
\begin{proof}   Using (\ref{exact-omega-mu}) we obtain
$$d_K(\varphi^*\theta)=\varphi^*\omega-\mu_\theta\circ\varphi\ \hbox{ in }A^*_K(P)\ ,
$$
hence $\pi^*_A(d_K(\varphi^*\theta))=\pi^*_A(\varphi^*\omega-\mu_\theta\circ\varphi)$. The right hand term coincides with $ \left\{\varphi^*(\omega)-\langle \mu_\theta\circ\varphi,\Omega_A\rangle\right\}^h$, and the left hand term coincides with $d(\varphi^*(\theta)^h)$ by   Proposition \ref{inverse}.
\end{proof}

Suppose  now that:
\begin{itemize}
\item $B$ and $F$ are complex manifolds and $A$ is a (1,1)-connection on $P$, i.e.,  the curvature $F_A\in A^2(B,\ad(P))$ has type (1,1). This condition is equivalent to the integrability of  the almost complex structure $J_A$ on the complexification $Q:=P\times_K K^\C$   (see \cite{LT2} Appendix 7.1). 
\item The $K$-action on $F$ is induced by a holomorphic action of the complex reductive group $K^\C$. 
\end{itemize}
If these conditions are satisfied, we define:
\begin{dt} A section $\varphi:P\to F$ is called $A$-holomorphic if one of the following equivalent conditions is satisfied:
\begin{itemize}
\item $(d\varphi)^h:\Gamma_A\to T_F$ commutes with the almost complex structure induced  via $\pi_*:\Gamma_A\to T_B$ on the horizontal distribution $\Gamma_A$ of $A$.
\item  the  $K^\C$-equivariant map $Q\to F$  induced by $\varphi$ is $J_A$-holomorphic.
\end{itemize}
\end{dt}

Suppose now that   $\omega=\frac{1}{4}dd^c\nu$, where $\nu$ is a $K$-invariant real function on $F$, and the operator  $d^c$ is defined as $d^c:=i(\bar\partial-\partial)=J^{-1} dJ$.  For a  tangent vector $X$ and a smooth function $\sigma$ one has $d^c(\sigma)(X)= -d\sigma(JX)$.

In this case we obtain an obvious integral of $\omega$ -- namely $\theta_\nu:=\frac{1}{4}d^c\nu$ -- hence also an associated exact moment map $\mu_{\theta_\nu}$ which we   simply denote  by  $\mu_\nu$. 

For an equivariant map $\varphi:P\to F$ we denote by $\nu(\varphi)\in A^0(B)$ the real function  defined by  $\nu(\varphi)\circ\pi=\nu\circ\varphi$. If  now $\varphi:P\to F$   is   $A$-holomorphic, we obtain
\begin{equation}\label{dc-commute}
\{\varphi^*(d^c\nu)\}^h=\pi^*(d^c(\nu(\varphi)))
\end{equation}
  by evaluating both sides on a horizontal tangent vector.
 
\begin{co}\label{ddc} With the notations and under the assumptions above suppose that $\omega=\frac{1}{4}dd^c\nu$, $A$ is a (1,1)-connection on $P$, and $\varphi$ is $A$-holomorphic.  Then
\begin{equation}\label{ddc-formula}
\frac{1}{4} dd^c(\nu(\varphi))=\left\{\varphi^*(\omega)-\langle \mu_\nu\circ\varphi,\Omega_A\rangle\right\}^h\ ,
\end{equation}
where the basic form on the right has been regarded as a 2-form on $B$.
\end{co} 

{\ }\\
{\bf Example:} Let $(F,h_F)$ be a Hermitian vector space, and $\nu_F:F\to\R$ the   $\U(F)$-invariant real function defined by $\nu_F(v):=h_F(v,v)$.  The Kähler form $\omega_F$ associated with  $h_F$ can be written as
$$\omega_F=\frac{i}{2}\partial\bar\partial \nu_F=\frac{1}{4} dd^c\nu_F\ .
$$ 
Put $\theta_F:=\frac{1}{4} d^c\nu_F$. Note that $\theta_F$ is $\U(F)$-invariant and 

\begin{equation}\label{theta-omega}
\omega_F=d\theta_F \ .
\end{equation}

Let now $\rho:K\to \U(F)$ be a unitary representation, and let $\langle\cdot ,\cdot \rangle_\kg$ be an $\ad$-invariant inner product on the Lie algebra $\kg$ of $K$. {\it The   standard moment map} $\mu_0:F\to\kg^\vee$ for the $K$-action on $F$ defined by $\rho$ is  the exact moment map associated with the $K$-invariant form $\theta_F$, i.e., it is given by
$$\langle \mu_0, a\rangle=-\iota_{a^\#}\theta_F\ \forall a\in\kg\ .
$$
A simple computation shows that
$$\mu_0(v)(a)=\left\langle\rho^*\left(-\frac{i}{2} v\otimes v^*\right),a \right\rangle_\kg=\left\langle -\frac{i}{2} v\otimes v^*,\rho_*(a)\right\rangle_{\u(F)}\ ,
$$
where the adjoint is computed with respect to $\langle\cdot ,\cdot \rangle_\kg$ and the standard inner product $\langle\cdot ,\cdot \rangle_{\u(F)}$ on the Lie algebra $\u(F)$ of skew-Hermitian endomorphisms of $F$.  
\begin{dt} \label{homotheties} A unitary representation $\rho:K\to U(F)$ is said to contain  the  homotheties of $F$  if there exists a $K$-invariant element $a_0\in \kg$ such that $\rho_*(a_0)=i\id_F$.
\end{dt}

For such an element $a_0$ one has
\begin{equation}\label{a0}\mu_0(v)(-2a_0)=\langle v\otimes v^*,\id_F\rangle_{i\ug(F)}=h_F(v,v)=\nu_F(v)\ .
\end{equation}

\subsection{Families of connections and sections}\label{families}

Let ${\cal C}$ be a manifold and ${\cal G}$ a Lie group with a free action on ${\cal C}$ such that the quotient manifold ${\cal B}:={\cal C}/{\cal G}$ exists and the projection ${\cal C}\to {\cal B}$ becomes a principal ${\cal G}$-bundle.  We allow both ${\cal C}$ and ${\cal G}$ to be infinite dimensional. Fix a connection $\Gamma_0$ in this principal bundle.
 
 Let $X$ be a closed $d$-dimensional manifold,  $P$ a principal $K$-bundle on $X$,  and 
$$u:{\cal G}\to\Aut(P)=\Gamma(X,P\times_\Ad K)$$
   a morphism of Lie groups. It is important  that $u$ is  not supposed to be surjective; in many interesting examples  ${\cal G}$ is the subgroup of $\Aut(P)$  associated with a normal subgroup of $K$.

\begin{dt} A smooth ${\cal G}$-equivariant  family of connections on $P$ is a smooth map   $\ag:{\cal C}\to {\cal A}(P)$  given by a family $(A_c)_{c\in{\cal C}}$ such that 
\begin{equation}\label{cond} A_{gc}=u(g)_*(A_c)\ \forall g\in {\cal G}\ \forall c\in{\cal C}\ .
\end{equation}
\end{dt}
{\ }\\
{\bf Examples:} 1. Let $E$, $E_0$ be Hermitian  vector bundles or rank $r$, $r_0$ on $X$ and $A_0$ a fixed connection on $A_0$. A pair of type $(E,E_0)$ is a pair $(A,\varphi)$ consisting of a Hermitian connection on $E$ and a section $\varphi\in A^0(X,E^\vee\otimes E_0)$.
In \cite{OT2} we gave a gauge theoretical construction of certain Quot spaces.  In this construction we used as configuration space  the space of irreducible pairs of type $(E,E_0)$, i.e., the subspace of pairs $(A,\varphi)\in {\cal A}(E)\times  A^0(X,E^\vee\otimes E_0)$ with trivial stabilizer with respect to the gauge group ${\cal G}:=\Gamma(X,\U(E))$.

In this case we have $K=\U(r)\times\U(r_0)$, $P:=P_E\times_X P_{E_0}$ is the product of the frame bundles associated with $E$ and $E_0$, and $u$ is  the obvious embedding of ${\cal G}$ in $\Aut(P)$. The map $\ag$ sends a pair $(A,\varphi)$ to the product connection $A\times A_0$. The configuration space and gauge group of the classical vortex equation is obtained as special case of this construction taking $(E_0,A_0)$ to be the trivial Hermitian line bundle endowed with the trivial connection.  \\

2. With the same notations as above let now $a_0$ be a fixed connection on the determinant $\U(1)$-bundle $\det(E)$. We denote by ${\cal A}_{a_0}(E)$ the space of $a_0$-oriented connections on $E$, i.e., the space of Hermitian connections on $E$ inducing $a_0$ on $\det(E)$. In the theory of master spaces \cite{OT1} one considers  as configuration space  the space of {\it irreducible $a_0$-oriented pairs} of type $(E,E_0)$, i.e., the subspace of pairs $(A,\varphi)\in{\cal A}_{a_0}(E)\times \Gamma(X,E^\vee\otimes E_0))$ with trivial stabilizer with respect to the natural action of the gauge group ${\cal G}:={\Gamma(X,\SU(E))}$. \\ \\

Put   $\Pg:={\cal C}\times P$ and note that  the  projection $\pi:\Pg\to {\cal C}\times X$ is a  principal $K$-bundle. We define  left ${\cal G}$-actions     on ${\cal C}\times X$ and $\Pg$ by 
$$ g (c,x):=(g c,x)\ ,\  g(c,p):=( gc,u(g) (p) )\ .$$

  The bundle $\Pg$ comes with a tautological connection $\Ag$  whose horizontal space at a point $(c,p)\in\Pg$ is $T_c {\cal C}\times \Gamma^{A_c}_p$. In other words $\Ag$  agrees with the trivial (product) connection in the ${\cal C}$ direction (in which $\Pg$ is trivial), and for every $c\in {\cal C}$  its restriction to $\{c\}\times P$ coincides with $A_c$. Consider points $c\in{\cal C}$, $x\in X$, $p\in P_x$ and tangent vectors
$\alpha_1$, $\alpha_2\in T_c{\cal C}$, $v_1$, $v_2\in T_x X$. Denote by $\tilde v_1$, $\tilde v_2\in T_{p}P$ the $A_c$-horizontal lifts of $v_1$, $v_2$. Then one has
\begin{equation}\label{CurvAg}\Omega_{\Ag}((\alpha_1,\tilde v_1),(\alpha_2,\tilde v_2))= \Omega_{A_c}(\tilde v_1,\tilde v_2)+\iota_{v_2}\ag_*(\alpha_1)-\iota_{v_1}\ag_*(\alpha_2)\  .\end{equation}

The equivariance property  (\ref{cond}) implies that $\Ag$ is ${\cal G}$-invariant. On the other hand, the fixed connection $\Gamma_0$  on the principal ${\cal G}$-bundle ${\cal C}\to{\cal B}$ defines a connection  $ \Gamma$ on the principal ${\cal G}$-bundle ${\cal C}\times X\to{\cal B}\times X$, whose horizontal distribution is  given by $  \Gamma_{c,x}:=\Gamma_{0,c}\oplus T_xX$.  The curvature of $ \Gamma$ is  
\begin{equation}\label{CurvTildeGamma}\Omega_{ \Gamma}((\alpha_1,v_1),(\alpha_2,v_2))=\Omega_{\Gamma_0}(\alpha_1,\alpha_2)\ \  \forall \alpha_1,  \ \alpha_2\in \Gamma_{0,c},\  \forall v_1, \ v_2\in T_xX\ .
\end{equation}

Denote by $\P$ the quotient bundle $\P:=\bar \Pg$ on the quotient manifold ${\cal B}\times X$. 
As explained in section \ref{quotConn},   the connection $\Gamma$ can be used to define the quotient connection $\A=\bar \Ag$ on the quotient $K$-bundle $\bar \pi:\P\to {\cal B}\times X$, whose curvature is given by Proposition \ref{CurvProp}. In order to write down explicitly this curvature, we need an explicit formula for the second term of the right hand side in (\ref{curv-formula}).
\begin{re}\label{Rem} Let $\nu\in \mathrm{Lie}({\cal G})$, $u_*(\nu)\in A^0(P\times_\ad\kg)$ the corresponding infinitesimal automorphism of $P$, $\nu^\#_\Pg$ the induced vector field on $\Pg={\cal C}\times P$, and let $\pg=(c,p)\in\Pg$. Then
$$\theta_{\Ag}(\nu^\#_{\Pg,\pg})=u_*(\nu)_p\ ,
$$
 where on the right $u_*(\nu)$ has been regarded as a $K$-equivariant map $P\to \kg$ and $u_*(\nu)_p\in\kg$ is the value of this map at $p$.
\end{re}
\begin{proof}
Indeed, $\nu^\#_{\Pg,\pg}=(\nu^\#_{{\cal C},c},u_*(\nu)^\#_p)=(\nu^\#_{{\cal C},c},(u_*(\nu)_p)^\#_p)$. Since $(\nu^\#_{{\cal C},c},0)$ is $\Ag$-horizontal  and $\theta_{\Ag}$ agrees with the connection form $\theta_{A_c}$ on $\{c\}\times P$, the claim follows from the properties of a connection form.
\end{proof}

 Applying Proposition \ref{CurvProp} to our situation we get  
\begin{pr} Let $c\in{\cal C}$, $[c]$ its image in ${\cal B}$. Let $a_1$, $a_2\in T_{[c]}{\cal B}$, and denote by $\alpha_i$  the corresponding $\Gamma_0$-horizontal lifts. Consider points $x\in X$, $p\in P_x$, tangent vectors $v_1$, $v_2\in T_xX$, and let $\tilde v_i \in \Gamma^{A_c}_p\in T_pP$ be their $A_c$-horizontal lifts. Put $\xi_i:=(a_i,v_i)\in T_{([c],x)}({\cal B}\times X)$ and denote by $\tilde \xi_i\in T_{[c,p]}\P$ their $\A$-horizontal lifts. Then
\begin{equation}\label{OmegamathbbA}\Omega_{\A}(\tilde\xi_1,\tilde\xi_2)=\Omega_{A_c}(\tilde v_1,\tilde v_2)+\iota_{v_2}\ag_*(\alpha_1)-\iota_{v_1}\ag_*(\alpha_2)+\left\{u_*\Omega_{\Gamma_0}(\alpha_1,\alpha_2)\right\}_p\ .
\end{equation}
\end{pr}
 \begin{proof} Indeed,  Proposition \ref{CurvProp} together with   Remark \ref{Rem} gives 
$$
\Omega_{\A}(\tilde\xi_1,\tilde\xi_2)=\Omega_{\Ag}((\alpha_1,\tilde v_1),(\alpha_2,\tilde v_2))+\theta_{\Ag} (\Omega_{\Gamma}((\alpha_1,v_1),(\alpha_2,v_2))^\#_{\Pg,(c,p)})$$
$$= \Omega_{\Ag}((\alpha_1,\tilde v_1),(\alpha_2,\tilde v_2))+\left\{u_*\Omega_{ \Gamma}((\alpha_1,v_1),(\alpha_2,v_2))\right\}_p \ .
$$
Now use (\ref{CurvAg}) and (\ref{CurvTildeGamma}).
\end{proof}

As a corollary we obtain   the following formula which computes  characteristic forms of degree 2 of the connection $\A$.
\begin{co} Let $h:\kg\times\kg\to\R$ be an $\ad$-invariant symmetric bilinear form. Let $c\in {\cal C}$, $x\in X$, $p\in P_x$, $a_1$, $a_2\in T_{[c]}{\cal B}$,  $v_1$, $v_2\in T_xX$, and let $\alpha_i \in \Gamma_{0,c}$, $\tilde v_i\in \Gamma_p^{A_c}$ be horizontal lifts of $a_i$, $v_i$ with respect to the connections $\Gamma_0$, $A_c$ respectively. Denote by $h(\Omega_\A\wedge \Omega_\A)\in A^4({\cal B}\times X)$ the characteristic form obtained from $\Omega_{\A}$ by coupling the wedge square  with $h$. Then:
$$-\frac{1}{2}h(\Omega_\A\wedge \Omega_\A)(a_1,a_2,  v_1,  v_2)=h(\ag_*(\alpha_1)\wedge\ag_*(\alpha_2))(v_1,v_2)-\ \ \ \ \ \ \ \ \ \ \ \ \ \ \ \ \ \ \ \ \ \ \ \ \ \ \ \ \ \  $$
$$\ \ \ \ \ \ \ \ \ \ \ \ \ \ \ \ \ \ \ \ \ \ \ \ \ \ \ \ \ \ \ \ \ \ \ \ \ \ \ \ \ \ \ \ \ \ \ \ \ \ \ \ \ \ -h(u_*\Omega_{\Gamma_0}(\alpha_1,\alpha_2)_p,\Omega_{A_c})(\tilde v_1,\tilde v_2)\ .
$$
\end{co}

\begin{proof} See \cite{DK}, \cite{LT2}.
\end{proof}

Note that the term $h(u_*\Omega_{\Gamma_0}(\alpha_1,\alpha_2)_p,\Omega_{A_c})(\tilde v_1,\tilde v_2)$ is independent of the choice of $p$ in the fibre $P_x$, so we can omit the symbol $p$ in this expression  and write $h(u_*(\Omega_\Gamma(\alpha_1,\alpha_2)),\Omega_{A_c})(  v_1, v_2)$. With this convention  the formula above yields the following identity in $A^2(X)$:
\begin{equation}\label{first}
-\resto{\frac{1}{2}h(\Omega_\A\wedge \Omega_\A)(a_1,a_2)}{\{[c]\}\times X}=h(\ag_*(\alpha_1)\wedge\ag_*(\alpha_2))-h(u_*\Omega_{\Gamma_0}(\alpha_1,\alpha_2),\Omega_{A_c})\ .
\end{equation} 
\vspace{4mm}

Let now be $F$ be a differentiable manifold endowed with a left $K$-action and let $E:=P\times_K F$ be the associated bundle with fibre $F$. Let $(\varphi_c)_{c\in {\cal C}}$ be a {\it  ${\cal G}$-equivariant family of sections in $E$ parameterized by our configuration space ${\cal C}$}, i.e., a smooth map $\phi: {\cal C}\to \Gamma(X,E)$ such that 
$$\varphi_{gc}=u(g)_*(\varphi)\  \ \forall g\in{\cal G},\ \forall c\in{\cal C}\ .
$$

Here we denote  by $u(g)_*:\Gamma(X,E)\to \Gamma(X,E)$ the push-forward map associated with $g$. Identifying  sections in $E$ with the corresponding $K$-equivariant maps $P\to F$, we can write
$$u(g)_*(\varphi)=\varphi\circ g^{-1}\ .
$$ 
Note that such a family $\phi$ defines a ${\cal G}$-invariant, $K$-equivariant map ${\cal C}\times P\to F$ given by $(c,p)\mapsto \varphi_c(p)$. This map descends to a $K$-equivariant map $\P={\cal C}\times P/{\cal G}\to F$, i.e.,  to a section $\Phi$ in the universal  $F$-bundle $\E:=\P\times_K F$ over ${\cal B}\times X$, which will be  called {\it the universal section of the family}.

Suppose now that $\omega_F$ is a symplectic form on $F$, the $K$-action on $F$ is symplectic  and admits a   moment map $\mu_F:F\to\kg^\vee$.

By  Proposition \ref{omega-mu} applied to the universal section $\Phi$ and the quotient connection $\A$ on the principal $K$-bundle $\bar\pi: \P\to {\cal B}\times X$ it follows that the basic form
$$
\left\{\Phi^*(\omega_F)-\langle \mu_F\circ \Phi,\Omega_\A\rangle\right\}^h\  $$
 is closed and represents the cohomology class
$$H_K(\bar \pi)^{-1}\circ H_K(\Phi)([\omega_F-\mu_F]_{\rm C})\in H^2_{\rm DR}({\cal B}\times X)\ .
$$
Let $\eta=\eta(\phi,\omega_F,\mu_F)\in Z^2_{\mathrm DR}({\cal B}\times X)$ be the {\it closed} form on ${\cal B}\times X$  which corresponds to this basic form. For  tangent vectors $a_i  \in T_{[c]}{\cal B}$,  horizontal lifts $\alpha_i\in\Gamma_{0,c}$  of $a_i$, and $p\in P$ we have
\begin{equation}\label{eta} \eta_{([c],x)} (a_1,a_2)=\omega_F(\Phi_*(\alpha_1), \Phi_*(\alpha_2))-\langle \mu_F(\varphi_c(p)),u_*\Omega_{\Gamma_0}(\alpha_1,\alpha_2)_p\rangle\ .
\end{equation}

For a fixed point $x\in X$ the term  $\langle\mu(\varphi_c(p)),u_*\Omega_{\Gamma_0}(\alpha_1,\alpha_2)_p\rangle$ is  independent of the choice of $p\in P_x$, so we can omit the symbol $p$ and  interpret the expression $\langle\mu(\varphi_c),u_*\Omega_{\Gamma_0}(\alpha_1,\alpha_2)\rangle$ as a smooth function on $X$. With this convention,   formula (\ref{eta}) above becomes
\begin{equation}\label{second}
\resto{\eta(a_1,a_2)}{\{[c]\}\times X}=\omega_F(\Phi_*(\alpha_1), \Phi_*(\alpha_2))-\langle \mu_F(\varphi_c),u_*\Omega_{\Gamma_0}(\alpha_1,\alpha_2)\rangle \ .
\end{equation}
\\ \\

Now we fix forms $w\in A^{d-2}(X)$, $v\in A^d(X)$,  and a $K$-invariant symmetric bilinear form $h$ on $\kg$. Define   ${\cal G}$-invariant  closed 2-forms  $\omega_v$, $\omega_w\in A^2({\cal C})$  by
$$\omega_w(\alpha_1,\alpha_2):=\int_X h(\ag_*(\alpha_1),\ag_*(\alpha_2))\wedge  w\ ,\ \omega_v(\alpha_1,\alpha_2):=\int_X \omega_F(\Phi_*(\alpha_1), \Phi_*(\alpha_2)) v\ .
$$
The  horizontal projections $\omega_w^h$, $\omega_v^h$ with respect to the connection $\Gamma$ are basic forms on the ${\cal G}$-bundle ${\cal C}\to{\cal B}$, so they can be interpreted as 2-forms on the base  ${\cal B}$ of this bundle. Note that general these forms are not closed.\\

We define  $\mu_{v,w}:{\cal C}\to \mathrm{Lie}({\cal G})^\vee$      by
$$\langle \mu_{v,w}(c),s\rangle:=\int_X h(u_*(s),\Omega_{A_c}\wedge  w)  +\langle \mu_F(\varphi_c), u_*(s) \rangle  v\ .
$$

When $w$ is closed, this map satisfies the axioms of a moment map for the ${\cal G}$-action on ${\cal C}$ with respect to the closed form  $\omega_{v,w}:=\omega_w+\omega_v$ which of course can be  degenerate in general.

Denote by $p_{{\cal B}*}: A^*({\cal B}\times X)\to  A^*({\cal B})$   fibre integration   associated with the projection $p_{\cal B}:{\cal B}\times X\to {\cal B}$.  Using formulae (\ref{first}), (\ref{second}) and the definition of $\mu_{v,w}$ we obtain  the following important identity:
\begin{equation}\label{id}
 \omega_{v,w}^h=p_{{\cal B}*}\left[-\frac{1}{2}[ h(\Omega_\A\wedge\Omega_\A)]\wedge p_X^*(w)+\eta\wedge p_X^*(v)\right]  +p_{{\cal B}^*}\langle \mu_{v,w},\Omega_{\Gamma_0}\rangle \ .\end{equation}

This proves:

\begin{thry} \label{omegahvw} The restriction of  the form  $\omega_{v,w}^h$ to the quotient ${\cal N}:=Z(\mu_{v,w})/{\cal G}$  is closed and is given by 
$$\resto{\omega_{v,w}^h}{{\cal N}}=p_{{\cal N}*}\left[-\frac{1}{2}[h( \Omega_\A\wedge\Omega_\A)]\wedge p_X^*(w)+\eta\wedge p_X^*(v)\right]\ .
$$
\end{thry}

By Corollary \ref{exact} we know that the form $\eta$ is exact when the moment map  $\mu_F$ is exact. Hence:
\begin{co}  \label{omegahvwexact} Suppose  $w$ is closed,  $\mu_F=\mu_0 +  \tau$, where $\mu_0$ is exact and  $\tau\in \kg^\vee$ is $K$-invariant. The one   has 
$$[\resto{\omega_{v,w}^h}{{\cal N}}]=p_{{\cal N}*}\left[ -\frac{1}{2} [h(\Omega_\A\wedge \Omega_\A] \cup  p_X^*[w]- [\tau(\Omega_\A)]\cup   p_X^*[v]\right] \  
$$
in $H^2_{\rm DR}({\cal N})$.
\end{co}

Note that each term of the right hand side of this formula can be written as the slant product of a characteristic class of the bundle $\P$ with a homology class of   $X$.

\section{Applications}\label{Applications}

\subsection{Kähler forms of   canonical Hermitian metrics  and tautological classes}
\label{KF}

\subsubsection{The general setting}

We now apply the results obtained in section \ref{families} to the  families of connections and sections associated with the universal gauge theoretical problem introduced in section \ref{Intro}. We will obtain the main theorems Theorem \ref{KformGen}  and Theorem \ref{Kform} stated in the introduction as corollaries of Theorem \ref{omegahvw} above.\\

Let $r:  K\to  K_0$ be an epimorphism of compact Lie groups with kernel   $N$. Fix a $K$-invariant inner product $k$ on $\kg$ and denote by $p_\ng:\kg\to\ng$ the $k$-orthogonal  projection onto the Lie algebra $\ng$ of $N$. The symmetric bilinear form  $h:\kg\times\kg\to\R$ defined by
$$h(x,y):=k(p_\ng (x),p_\ng(y))
$$
is $K$-invariant and non-degenerate on $\ng$. Suppose we are given a left  $K$-action on a Kähler manifold $(F,J_F,g_F)$  by holomorphic isometries,  and  a moment map $\mu_F$ for this action on the corresponding symplectic manifold $(F,\omega_F)$. We denote by $m_F:F\to  \kg$ the $ \kg$-valued map defined by $\mu_F=k(m_F,\cdot)$. 

Let now $(X,J,g)$ be a  $n$-dimensional  Kähler manifold,  $\pi: P\to X$ a   principal $  K$-bundle on $X$ and $P_0:=  P\times_K   K_0$ the associated $K_0$-bundle.  We fix a connection $A_0\in{\cal A}( P_0)$ of type $(1,1)$ (our ``orientation data").

Let   $E:=  P\times_{K} F$ be the associated bundle with fiber $F$. The gauge group of our moduli problem is 
$${\cal G}:=\Gamma (X,  P\times_{K} N)$$
 and acts  from the left on the space of $A_0$-oriented pairs ${\cal A}_{A_0}(  P)\times \Gamma(X,E)$. This space has a natural Kähler metric depending on the triple $(k,g,g_F)$ whose Kähler form is given by the formula (\cite{LT2} p. 66):
  $$\Omega((\alpha,\psi),(\beta,\chi):=\int_X h(\alpha\wedge\beta)\wedge\omega_g^{n-1}+(n-1)!\int_X \omega_F(\psi,\chi)\vol_g\  .
  $$

In order to avoid technical problems we will concentrate  on the configuration space   of irreducible oriented pairs. Let ${\cal C}^*:=[{\cal A}_{A_0}(P)\times \Gamma(X,E)]^*$ be  the  open subspace of  irreducible oriented pairs , i.e., of pairs 
$(A,\varphi)\in {\cal A}_{A_0}(P)\times \Gamma(X,E)$ with trivial stabilizer with respect to the ${\cal G}$-action.    After suitable Sobolev completions the projection ${\cal C}^*\to {\cal B}^*:={\cal C}^*/{\cal G}$ becomes a principal ${\cal G}$-bundle. We endow this bundle with the connection $\Gamma_0$ defined by $L^2$-orthogonality  to the ${\cal G}$-orbits\footnote{The construction can be extended to the non-Kählerian framework and yields a strongly KT Hermitian metric on the moduli space, but one has to use a different connection $\Gamma_0$ \cite{LT1}, \cite{LT2}.  } \cite{LT2}.

Note that we have a  tautological family $(A_c)_{c\in {\cal C}^*}$ of connections, and  a  tautological  family   $(\varphi_c)_{c\in {\cal C}^*}$ of sections parameterized by ${\cal C}^*$.  The corresponding maps $\ag$, $\phi$ are given by the projections onto the two factors.\\

The gauge theoretical   moduli space  ${\cal M}^*$ corresponding to the data above is the space of  equivalence classes of pairs $(A,\varphi)\in{\cal C}^*$ satisfying the generalized  vortex  equations
$$\left\{
\begin{array}{ccc}
F_A^{02}&=&0\ \ \ \\\ 
\varphi  \hbox{ is }  A\hbox{-holomorphic }\\
p_{\ng}\left[  \Lambda_g  F_A+ m_F(\varphi)\right] &=&0\ .
\end{array}
\right. \eqno{(V)}
$$
In the third equation  we used the same symbol  for the bundle epimorphism   $P\times_{ K}  \kg\to P\times_{  K}\ng$ induced by $p_\ng:\kg\to\ng$. \\

 Taking $w=\omega_g^{n-1}$ and $v=(n-1)! \vol_g=\frac{1}{n} \omega_g^n$, we  obtain for $s\in\mathrm{Lie}({\cal G})=A^0(P\times_K\ng)$
$$\langle \mu_{v,w}(A,\varphi)),s\rangle=\int_X h(s,F_A\wedge\omega_g^{n-1})+ (n-1)! \langle \mu_F(\varphi),  s\rangle \vol_g$$
$$\ \ \ \ \ \ \ \ \ \ \ \ \ \ \ \ \ \ \ \ \ \ \ \ \ \ \ =\int_X h(s,F_A\wedge\omega_g^{n-1})+ (n-1)!  k(s,m_F(\varphi))\vol_g
$$
$$\ \ \ \ \ \ \ \ \ \ \ \  \  \  \ \ \ \ \ \ \ \ \ \ \ \ \ \ \ \ \ =(n-1)! \int_X  \left\{h(s,p_\ng \Lambda_g F_A)+ k(s, p_\ng(m_F(\varphi))\right\}\vol_g $$
$$\ \ \ \ \ \ \ \ \ \ \ \ \  \ \ \   \ \ \ \ \ \ \ \ \ \ \ \ \ \ \  =(n-1)! \int_X  \left\{h(s, p_\ng \Lambda_g F_A)+ h(s,p_\ng(m_F(\varphi))\right\}\vol_g $$
$$\ \ \ \ \ \ \ \ \ \ \ \ \ \ \ \ \ \ \ =(n-1)! \int_X  h(s, p_\ng (\Lambda_g F_A+m_F(\varphi)) ) \vol_g .
$$
This shows that the third equation in $(V)$ is equivalent to the vanishing of the moment map $\mu_{v,w}$ introduced above. On the other hand, using the fact that  the restriction of $h$ to $\ng\times\ng$ is an inner product one can prove easily that the  restriction of the form $\omega^h_{v,w}$ on ${\cal M}^*$ is non-degenerate. This form coincides with  the Kähler form of the standard Hermitian metric $g_{{\cal M}^*}$ on ${\cal M}^*$ associated with the pair $(\resto{h}{\ng\times\ng},g_F)$ (see \cite{LT2}).
Therefore, using Theorem \ref{omegahvw} and Corollary \ref{omegahvwexact}   we obtain
 
\begin{thry} \label{KformGen} The canonical Hermitian metric on ${\cal M}^*$ is Kähler and its Kähler form is
$$\omega_{\cal M^*}=p_{{\cal M}^*,*}\left[-\frac{1}{2}[ h(\Omega_\A\wedge\Omega_\A)]\wedge p_X^*(\omega_g^{n-1})+(n-1)!\eta\wedge p_X^*(\vol_g)\right]\ ,
$$
with $\eta:=\left\{\Phi^*(\omega_F)-\langle \mu_F\circ\Phi,\Omega_\A\rangle\right\}^h$, which is a closed 2-form. If $\mu_F$ can be written as $\mu_F=\mu_0+\tau$, where $\mu_0$ is an exact moment map and  $\tau\in   \kg^\vee$ is   $K$-invariant, then the de Rham cohomology class $[\omega_{\cal M^*}]_{\mathrm {DR}}$ is given by
$$[ \omega_{\cal M^*}]_{\mathrm {DR}}=p_{{\cal M}^*,*}\left[ -\frac{1}{2} [h(\Omega_\A\wedge \Omega_\A)] \cup  p_X^*[\omega_g]^{n-1}-(n-1)! [\tau(\Omega_\A)]\cup   p_X^*[\vol_g]\right] \ . 
$$
\end{thry}

Note that, by formula (\ref{second})  the term  $\tau(\Omega_\A)$ coincides with $\tau(\Omega_{\Gamma_0})$, hence this term depends only on the restriction $\resto{\tau}{\ng}$.

\begin{thry}  \label{Kform}
Let  $F$ be a Hermitian vector space, and  $\rho: K\to \U(F)$ a unitary representation.  For a $ K$-invariant linear form $\tau\in  \kg^\vee$ put  $\mu_F:=\mu_0+\tau$, where $\mu_0$ is the standard moment map for the $  K$-action on $F$.

Suppose  there exists a $K$-invariant element $a_0\in \ng$ such that $\rho_*(a_0)=i \id_F$. Then the Kähler form $\omega_{\cal M^*}$ is given by

$$ \omega_{\cal M^*}=p_{{\cal M}^*,*}\left[ -\frac{1}{2}  h(\Omega_\A\wedge \Omega_\A)  \wedge  p_X^*(\omega_g)^{n-1}-(n-1)!  \tau(\Omega_\A) \wedge   p_X^*(\vol_g)\right] \ ,
$$ 
hence it  coincides with fiber integration of a characteristic form  of the connection $\A$ on the universal bundle $\P$. \end{thry}

\begin{proof}

Since  $\mu_F=\mu_0+\tau$ and the form $\eta$ depends linearly on $\mu_F$ one has
$$ \omega_{\cal M^*}-p_{{\cal M}^*,*}\left[ -\frac{1}{2}  h(\Omega_\A\wedge \Omega_\A)  \wedge  p_X^*(\omega_g)^{n-1}-(n-1)!  \tau(\Omega_\A) \wedge   p_X^*(\vol_g)\right]  $$
$$=(n-1)! p_{{\cal M}^*,*}\left(\eta_0\wedge p_X^*(\vol_g) \right)\ ,
$$
where
$$\eta_0=\left\{\Phi^*(\omega_F)-\langle \mu_0\circ\Phi,\Omega_\A\rangle\right\}^h \ .
$$

By Corollary \ref{ddc} we obtain on ${\cal B}^*\times X$ the identity
$$\eta_0=\frac{1}{4}dd^c | \Phi|^2\ ,
$$
where the universal section $\Phi$ has been regarded as  a section in the Hermitian vector bundle $\E:=\P\times_{K} F$  on ${\cal B}^*\times X$. Therefore it suffices to prove that 
$$p_{{\cal M}^*,*} (dd^c | \Phi|^2\wedge p_X^*(\vol_g))=0\ .$$
Since the operators  $d$, $d^c$ commute with the fibre integration associated with a {\it holomorphic} map we have
$$p_{{\cal M}^*,*} (dd^c | \Phi|^2\wedge p_X^*(\vol_g))=dd^c\left\{ p_{{\cal M}^*,*}   (| \Phi|^2\wedge p_X^*(\vol_g))\right\} \ .
$$
 Since   $a_0\in\ng$ is $K$-invariant by assumption, it defines a section $\tilde a_0\in A^0(P\times_K\ng)$. Taking the $L^2$ inner product of   the third equation in $(V)$ with $\tilde a_0$ we compute
$$0=\langle\Lambda F_A,\tilde a_0\rangle_{L^2}+ \langle m_0(\varphi),\tilde a_0\rangle_{L^2}+ \tau(a_0)\Vol_g\ .
$$
The first term is a characteristic number of $P$, so is a constant.  On the other hand by  formula  (\ref{a0}) we obtain:
$$\langle m_0(\varphi),\tilde a_0\rangle_{L^2}=-\frac{1}{2}\int_X |\varphi|^2 \vol_g\ .
$$
   This implies that the map 
$$[A,\varphi]\mapsto \int_X | \varphi|^2
$$
is constant on ${\cal M}$. But this map coincides with $p_{{\cal M}^*,*}  ( |\Phi|^2\wedge p_X^*(\vol_g))$.
 \end{proof}
\subsubsection{Vortex moduli spaces} Let $E$,  $E_0$ be Hermitian vector bundles on $X$ of ranks $r$,   $r_0$, $P_E$, $P_{E_0}$ the associated frame bundles, and $t\in\R$. Fix a Hermitian connection $A_0$ of type $(1,1)$ on $E_0$, and let  ${\cal M}_t={\cal M}_t(E, E_0,A_0)$ be  the  moduli space   of pairs $(B,s)\in {\cal A}(E)\times A^0(\Hom(E, E_0))$ solving the $t$-vortex equation:
$$\left\{
\begin{array}{ccc}
F_B^{02}&=&0\\
\bar\partial_{B,A_0} s&=&0\\
i\Lambda F_{A}-\frac{1}{2} s^*\circ s + t\id_E&=&0
\end{array}\right. \eqno{(V_t(E,E_0,A_0))}
$$
 modulo the gauge group ${\cal G}:=\Gamma(X,\U(E))$.  
This moduli problem can be obtained as a special case of  the universal moduli problem $(V)$ considered in Theorem \ref{Kform} by taking $K=\U(r)\times \U(r_0)$, $N=\U(r)$, $F:= \Hom(\C^r,\C^{r_0})$ with the obvious unitary representation $\rho:K\to \U(F)$, with bundle  $P:=P_E\times_X P_{E_0}$, and  moment map
$$m_F(f)=(\frac{i}{2} f^*\circ f -it\id_{\C^r},-\frac{i}{2} f\circ f^* )\ .$$  

Note that the space ${\cal A}_{A_0}(P)$ can be identified with the space ${\cal A}(E)$ of Hermitian connections on $E$. We put ${\cal C}^*:=[{\cal A}(E)\times A^0(\Hom(E, E_0))]^*$ and use the quotient construction to define a universal connection $\B$ on the universal vector bundle 
$$\E:=\qmod{{\cal C}^*\times E}{{\cal G}}$$
 over ${\cal B}^*\times X$.

We use the standard inner product   $k((a,a_0),(b,b_0))=-\tr(ab)-\tr(a_0b_0)$ on   $\u(r)\oplus\u(r_0)$. With this choice we get $h((a,a_0),(b,b_0))=-\tr(ab)$ and $\mu_F=\mu_0+\tau$, where $\tau(a,a_0)=it \tr(a)$. This yields
$$  \tau(\Omega_\A)=2\pi tc_1(\B)\ ,\ h(\Omega_\A\wedge \Omega_\A)=-\tr(\Omega_\B\wedge\Omega_\B)=4\pi^2(c_1(\B)^2-2c_2(\B))=8\pi^2 \mathrm{ch}_2 \ ,
$$
where $\mathrm{ch}_2:=\frac{1}{2} c_1^2-c_2$ stands for the second component of the Chern character. Theorem \ref{Kform} gives

\begin{pr} \label{KFV} The  natural Kähler form on the moduli space 
$${\cal M}^*_t={\cal M}^*_t(E, E_0,A_0)$$
 of irreducible solutions of $(V_t)$ is given by
\begin{equation}\label{Kformvortex}  \omega_{\cal M^*_t}=p_{{\cal M}^*_t,*}\left[-4\pi^2\mathrm{ch}_2(\B)  \wedge  p_X^*(\omega_g^{n-1})-2t\pi(n-1)!  c_1(\B) \wedge   p_X^*(\vol_g)\right]   , 
\end{equation}
where $\B$  denotes the universal connection on the universal bundle $\E:= {{\cal C}^*\times E}/{{\cal G}}$ on ${\cal M}^*_t\times X$.

\end{pr}

\begin{re} Let ${\cal L}_0$ be a holomorphic line bundle on $X$ of Chern class $l_0\in\NS(X)$, $\chi_0$  a Hermite-Einstein metric on ${\cal L}_0$, $a_0$ the corresponding Hermite-Einstein connection on the Hermitian line bundle $(L_0,\chi_0)$, and $\gamma_0$ the first Chern form of $a_0$. Then the moduli spaces
${\cal M}_t^*(E,E_0,A_0)$ and ${\cal M}^*_{t-2\pi\Lambda \gamma_0}(E\otimes L_0,E_0\otimes L_0,A_0\otimes a_0)$ can be identified via the map $[B,s]\mapsto [B\otimes a_0,s\otimes\id_{L_0}]$ and the  Kähler forms given by formula (\ref{Kformvortex}) correspond via this identification. 
\end{re}

\begin{proof} Let ${\cal C}^*$, $\tilde {\cal C}^*$ be the configuration spaces of irreducible pairs associated with the two data systems $(E,E_0)$, $(E\otimes L_0,E_0\otimes L_0)$ respectively, ${\cal B}^*$, $\tilde {\cal B}^*$ the corresponding  quotients,  $\E\to {\cal B}^*\times X$,  $\tilde\E\to\tilde{\cal B}^*\times X$  the corresponding universal bundles, and $\B$, $\tilde\B$ the corresponding universal connections.  Using the formula
$$i\Lambda F_{B\otimes a_0}=i\Lambda F_{B}+ 2\pi\Lambda\gamma_0
$$
we see that the map $j:{\cal C}^*\to\tilde{\cal C}^*$ given by $j(B,s)=(B\otimes a_0,s\otimes\id_{L_0})$ maps  the space of solutions of the equation $(V_t(E,E_0,A_0))$ bijectively  onto the space of solutions of  the equation $(V_{t-2\pi\Lambda \gamma_0}(E\otimes L_0,E_0\otimes L_0,A_0\otimes a_0))$. We have
$$([j]\times\id_X)^*(\tilde \E)=\E\otimes p_X^*(L_0)\ , \ ([j]\times\id_X)^*(\tilde \B)=\B\otimes p_X^*(a_0)\ ,
$$
$$([j]\times\id_X)^*(c_1(\tilde \B))=c_1( \B)+rp_X^*(\gamma_0)\ ,
$$
$$([j]\times\id_X)^*(2c_2(\tilde\B)-c_1^2(\tilde\B))=(2c_2(\B)-c_1^2(\B))- 2c_1(\B)\wedge p_X^*(\gamma_0)-r p_X^*(\gamma_0^2)\ .
$$
The statement follows now by direct computation.
\end{proof}

Let $\bar{\cal A}(E)$ be the space of semi-connections on $E$,  $\bar{\cal C}^s:=[\bar{\cal A}(E)\times A^0(\Hom(E,E_0))]^s$  the space of {\it simple pairs of type} $(E,{\cal E}_0)$ endowed with the natural action of the complex gauge group ${\cal G}^\C:=\Gamma(X,\GL(E))$, and $\bar{\cal B}^s$ the quotient 
$$\bar{\cal B}^s:=\qmod{\bar{\cal C}^s}{{\cal G}^\C}
$$
 (see \cite{OT2} p. 559 for details). After suitable Sobolev completions (which will again be omitted) this quotient becomes an infinite dimensional complex space. As in the real gauge theoretical framework we obtain a universal bundle 
 $$\mathscr{E}:=\qmod{\bar{\cal C}^s\times E}{{\cal G}^\C}
 $$
of rank $r$ on the product $\bar{\cal B}^s\times X$.   Fixing a holomorphic structure ${\cal E}_0$ on $E_0$ we obtain a finite dimensional complex subspace 
$${\cal M}^s={\cal M}^s(E,{\cal E}_0)\subset \bar{\cal B}^s\ ,$$
called the moduli space of simple holomorphic pairs of type $(E,{\cal E}_0)$ and the restriction of $\mathscr{E}$ to ${\cal M}^s\times X$ is holomorphic.

Choosing ${\cal E}_0$ to be the  holomorphic structure on $E_0$ defined by the semi-connection $\bar\partial_{A_0}$, we obtain  a map $KH:{\cal B}^*\to \bar{\cal B}^s$
induced by $(B,s)\mapsto ({\cal E}_{\bar\partial_B},s:{\cal E}_{\bar\partial_B}\to{\cal E}_0)$.

Note that one has an obvious isomorphism 
\begin{equation}\label{UnivBdls}
(KH\times\id_X)^*(\mathscr{E})\simeq\E 
\end{equation}
and an obvious inclusion $KH({\cal M}^*_t)\subset  {\cal M}^s$. The image of this inclusion is given by the Kobayashi-Hitchin correspondence for twisted holomorphic pairs (see Theorem 2.7 \cite{OT2}):
\begin{thry} \label{Mst}  The map $(B,s)\mapsto ({\cal E}_{\bar\partial_B},s:{\cal E}_{\bar\partial_B}\to{\cal E}_0)$ induces a real analytic isomorphism
$$KH:{\cal M}^*_t(E,E_0,A_0)\textmap{\simeq} {\cal M}^\st_{\tg} (E,{\cal E}_0)
$$
onto the moduli space of $\tg$-stable pairs  of type $(E,{\cal E}_0)$ with $\tg:= \frac{(n-1)!\Vol_g(X)}{2\pi} t$.
 \end{thry}

In order to save on notations we will denote by the same symbol $\omega_t$ the Kähler form $\omega_{{\cal M}^*_t}$  defined above and  also its image   under the isomorphism $KH$.  \\  

Let   $G\to \Aut({\cal E}_0)$ be a morphism of Lie groups. Then $G$ acts naturally on the  space $\bar{\cal A}(E)\times A^0(\Hom(E,E_0))$ of pairs of type $(E,{\cal E}_0)$ leaving invariant the holomorphy, the simplicity,  and the $\tg$-stability conditions. Since this action commutes with the action of the complex gauge group ${\cal G}^\C$, it induces a $G$-action on $\bar{\cal B}^s$ leaving invariant the finite dimensional subspaces $ {\cal M}^s$ and ${\cal M}^\st_{\tg} (E,{\cal E}_0)$.

Our Proposition  \ref{KFV} combined with the Kobayashi-Hitchin  correspondence given by Theorem \ref{Mst}  and the isomorphism (\ref{UnivBdls}) has an important consequence: the  Kähler class $\omega_{t}$    has a   natural  lift   in equivariant cohomology.
\begin{thry} \label{Symm}  
For any morphism of Lie  groups $G\to \Aut({\cal E}_0)$ the class
$$ \left\{-4\pi^2\mathrm{ch}_2^G(\mathscr{E})  \cup  p_X^*[\omega_g^{n-1}]-2t\pi(n-1)!  c_1^G(\mathscr{E}) \cup   p_X^*[\vol_g]\right\}/[X] 
$$
is a lift of  the Kähler class $[\omega_{t}]$ to $H^2_G({\cal M}^\st_{\tg} (E,{\cal E}_0),\R)$. 
\end{thry}

\begin{proof} Put ${\cal M}^\st_{\tg}:={\cal M}^\st_{\tg} (E,{\cal E}_0)$. Using the trivial action on the second factor the product ${\cal M}^\st_{\tg}\times X$ becomes a $G$-space, and one has  natural isomorphisms 
$$H^*_G({\cal M}^\st_{\tg}\times X,\R)\simeq H^*(EG\times_G({\cal M}^\st_{\tg}\times X),\R)\simeq H^*((EG\times_G{\cal M}^\st_{\tg})\times X),\R)$$
$$\simeq H^*_G({\cal M}^\st_{\tg},\R)\otimes H^*(X,\R)\ .$$
Therefore the slant product with the   fundamental class $[X]$ is well defined on both cohomology algebras $H^*({\cal M}^\st_{\tg}\times X,\R)$, $H^*_G({\cal M}^\st_{\tg}\times X,\R)$, and defines a  commutative diagram \vspace{-3mm}
\begin{diagram}[s=8mm,w=8mm,midshaft]
  H^j_G({\cal M}^\st_{\tg}\times X,\R) & \rTo^{} & H^j({\cal M}^\st_{\tg}\times X,\R)\ \ \ \\
\dTo<{/[X]} & & \dTo<{/[X]} \\
  H^{j-2n}_G({\cal M}^\st_{\tg},\R) & \rTo^{ } & H^{j-2n}({\cal M}^\st_{\tg},\R)   
\end{diagram}
for any $j\in\N$. Since the equivariant Chern classes  $c_k^G(\mathscr{E})$ are   lifts of the classes $c_k(\mathscr{E})$ to equivariant cohomology, and fiber integration $p_{{\cal M}^*_t,*}$ induces the morphism $/[X]$ in de Rham cohomology, the result follows  from formula (\ref{Kformvortex}).

\end{proof}

In \cite{OT2} we have also shown that, for sufficiently large $\tg$, the moduli space ${\cal M}^\st_{\tg} (E,{\cal E}_0)$ can be identified with the Quot space ${\cal Q}uot^E_{{\cal E}_0}$ of quotients of ${\cal E}_0$ with locally free kernel of ${\cal C}^\infty$-type $E$. 
\begin{re}\label{new-Rem}
The natural  identification ${\cal M}^\st_{\tg} (E,{\cal E}_0)={\cal Q}uot^E_{{\cal E}_0}$ (for $\tg\gg 0$) is equivariant with respect to the $\Aut({\cal E}_0)$-actions on the two spaces.
Via this identification the restriction of the universal bundle $\mathscr{E}$ to ${\cal Q}uot^E_{{\cal E}_0}\times X$  coincides (as an $\Aut({\cal E}_0)$-bundle) with the universal kernel associated with this Quot space.
\end{re}

 In order to save on notations we  denote the  restrictions of the  universal bundles $\mathscr{E}$ and $\E$ to ${\cal M}^\st_{\tg} (E,{\cal E}_0)\times X$ and ${\cal M}^*_t(E,E_0,A_0)\times X$ respectively by the same symbols. Using Theorem \ref{Mst} and the isomorphism (\ref{UnivBdls}) we obtain 
\begin{re} \label{universalK} Suppose that $t$ is sufficiently large such that the isomorphism $KH$ defines a real analytic isomorphism
$$KH:{\cal M}^*_t(E,E_0,A_0)\to {\cal Q}uot^E_{{\cal E}_0}\ .
$$
Via this isomorphism the underlying ${\cal C}^\infty$-bundle of the pull-pack of the universal kernel $\mathscr{E}$ on   $ {\cal Q}uot^E_{{\cal E}_0}\times X$ can be identified with the universal bundle $\E$ on ${\cal M}^*_t(E,E_0,A_0)\times X$.
\end{re}
\vspace{3mm}

In the case $r=1$ the identification ${\cal M}^\st_\tg(E,{\cal E}_0)={\cal Q}uot^E_{{\cal E}_0}$ holds for $\tg> -\frac{\deg_g(E)}{\rk(E)}$ \cite{OT2}, and we get a more explicit formula for the Kähler class $[\omega_{{\cal M}^*}]$.  Denote by ${\cal B}(E)$ the space of gauge classes of Hermitian connections on the Hermitian line bundle $E$. Recall that ${\cal B}(E)$ can be identified with the total space of a vector bundle over the torus of Yang-Mills connections on $E$ (which can be identified via the classical Kobayashi-Hitchin correspondence with $\Pic^{c_1(E)}(X)$). 

Therefore   we have a well defined isomorphism $\delta:H_1(X,\Z)/\Tors\to H^1({\cal B}(E),\Z)$, which can be regarded as an element  of  $H^1({\cal B}(E),\Z)\otimes H^1(X,\Z)$. Denote by $\nu$ the natural map ${\cal B}^*\to{\cal B}(E)$. The Künneth decomposition of $c_1(\E)$ is  
 \begin{equation}\label{cAbelian}
 c_1(\E)=-\gamma\otimes 1+(\nu\times\id_X)^*\delta+1\otimes c_1(E)\ ,
 \end{equation}
where $\gamma=-c_1(\resto{\E}{{\cal B}^*\times\{x_0\}})$  for a fixed point $x_0\in X$. Let $(h_i)$ be a basis of the free $\Z$-module $H_1(X,\Z)/\rm Tors$,  $(h^i)$ the dual basis of $H^1(X,\Z)$ and 
$$h_{ij}:=\langle  h^i\cup h^j\cup [\omega_g^{n-1}],[X]\rangle\ .$$
  Then
\begin{equation}\label{c2Abelian}[p_{{\cal B}^*}]_* (c_1^2(\E)\cup p_X^*([\omega_g^{n-1}])=-2\nu^*(\theta) - 2\deg_g(E)\gamma\ ,
\end{equation}
where  
$$\theta:=\sum_{i<j} h_{ij}\   \delta(h_i)\cup \delta(h_j)\ .$$
 With this notations we get
\begin{equation}\label{Perutz}[\omega_{{\cal M}^*_t}]=\resto{\left\{4\pi^2\big(\nu^*(\theta)+\deg_g(E)\gamma\big)+ 2t\pi(n-1)! \mathrm{Vol}_g(X)\gamma\right\}}{{\cal M}^*_t} \ .
\end{equation}
\def\st{\mathrm{st}}
In order to save on notations, from now on we agree to denote by the same symbols the restrictions of the classes $\theta$, $\gamma$ to $\Pic^{c_1(E)}(X)$ and ${\cal M}^*_t$ respectively. We also agree to use the same symbol $\theta$ for the pull-back of this class via $\nu$. With this convention we obtain the simple formula 
\begin{equation}\label{Perutz-new}\frac{1}{4\pi^2}[\omega_{{\cal M}^*_t}]= \theta+(\deg_g(E)+\tg)\gamma \ ,
\end{equation}
which specializes to a formula obtained by Perutz (see Theorem 3 \cite{P}) for $r_0=1$ and $n=1$.  

\subsection{Volumina of Abelian Quot spaces}

Let $X$  be compact complex manifold, let ${\cal E}_0$ be a locally free sheaf of rank $r_0$ on $X$,  and  let $E$ be a differentiable vector bundle of rank $r$. As above we denote by ${\cal Q}uot^{E}_{{\cal E}_0}$ the Quot space of quotients  $q:{\cal E}_0\to Q$  of ${\cal E}_0$ with locally free kernel ${\cal C}^\infty$-isomorphic to $E$. Since locally-freeness is an open condition in flat families it follows that ${\cal Q}uot^{E}_{{\cal E}_0}$ is an open subspace of the Douady space ${\cal Q}uot_{{\cal E}_0}$. This Quot space is characterized by obvious universal properties.

For $r=1$ the bundle $E$ is a differentiable line bundle, hence its differentiable isomorphism type is determined by its Chern class. Therefore in this case is convenient to fix 
 $m\in\NS(X)$ and to denote by ${\cal Q}uot^{m}_{{\cal E}_0}$ the Quot space of   quotients  $q:{\cal E}_0\to Q$  of ${\cal E}_0$ with locally free kernel of rank 1 and Chern class $m$.
 We will give an explicit description  of certain Quot spaces of this type. We begin with a simple flatness criterion:

\begin{lm}\label{flat}
Let $\pi:{\cal X}\to B$ be a morphism of complex spaces, and let ${\cal F}\textmap{j} {\cal E}$ be  a morphism of coherent sheaves on ${\cal X}$ such that ${\cal E}$ is flat over $B$. Then the following two conditions  are equivalent:
\begin{enumerate}[(i)]
\item For every $b\in B$ the restriction $j_b:{\cal F}_{X_b}\to {\cal E}_{X_b}$
 of $j$ to the fibre $X_b$ of $\pi$ is a sheaf monomorphism.
 \item $j$ is a monomorphism and the quotient ${\cal Q}:={\cal E}/j({\cal F})$ is flat over $B$.
\end{enumerate}
\end{lm}
 
\begin{proof} For $b\in B$ and $x\in X_b$ the natural morphism
$$j_{x}^b :{\cal F}_x\otimes_{{\cal O}_b}\C\to {\cal E}_x\otimes_{{\cal O}_b}\C \  
$$
induced by $j$ is just the morphism induced by $j_b$ between these stalks. Therefore 
 (i) is equivalent to the injectivity of $j_{x}^b$ for every  $b\in B$, $x\in X_b$.
Fix $b\in B$, $x\in X_b$ and note that ${\cal E}_x$ is a flat ${\cal O}_b$-module by assumption. By Corollary 1.5 p. 165 \cite{BS} $j_{x}^b$ is a monomorphism if and only if $j_x:{\cal F}_x\to {\cal E}_x$ is a monomorphism and $\mathrm{coker}(j_x)$ is a flat ${\cal O}_b$-module.

\end{proof}

\begin{re}\label{smooth-fibres} The hypothesis of Lemma \ref{flat} and condition (i) of this lemma are satisfied when $\pi$ is a flat morphism,  ${\cal F}$, ${\cal E}$ are locally free, and for any $b\in B$ the following two conditions hold:
\begin{enumerate}[(i)]
\item the fiber  $X_b$ is smooth, 
\item the restriction $j_b:{\cal F}_{X_b}\to {\cal E}_{X_b}$ is a generically injective morphism of vector bundles over $X_b$.
\end{enumerate}
Therefore, by Lemma \ref{flat}, in this case ${\cal Q}:={\cal E}/j({\cal F})$ is flat over $B$.
\end{re}
Note that the statement holds even if the total space ${\cal X}$ is not reduced.\\

Now we come back to our situation: let $X$ be a compact complex manifold.
\begin{dt} A pair $(m,{\cal E}_0)$  consisting of a class  $m\in\NS(X)$ and   a holomorphic bundle ${\cal E}_0$ of rank $r_0$ on $X$ is called acyclic if
$$h^i({\cal L}^\vee\otimes{\cal E}_0)=0 \ \ \forall i>0\ ,\ \forall [{\cal L}]\in\Pic^m(X) \ .$$
\end{dt}
{\ }\\
{\bf Examples:} 1. Suppose that $X$ is a smooth projective manifold, $H$ an ample divisor on $X$ and   $(m,{\cal E}_0)$ is a pair consisting of a class $m\in\NS(X)$ and a holomorphic vector bundle ${\cal E}_0$.  Using Serre's vanishing theorem and  the compactness of $\Pic^m(X)$ one shows easily that  for all sufficiently large $k\in\N$ the pair $(m,{\cal E}_0(kH))$ is acyclic.
\\ \\
2. Suppose now that $X$ is a curve, $m\in\Z$  and ${\cal E}_0$ is a polystable bundle on $X$ with  $\deg({\cal E}_0)>  r_0 m +2 r_0 (g(X)-1)$. Then for every $[{\cal L}]\in\Pic^m(X)$ the bundle ${\cal K}\otimes {\cal E}_0^\vee\otimes {\cal L}$ admits a Hermite-Einstein metric with negative Einstein constant, so that $H^0({\cal K}\otimes {\cal E}_0^\vee\otimes {\cal L})=0$. Using Serre duality we see that the pair $(m,{\cal E}_0)$ is acyclic.
\\ \\
   
Let $\mathscr{L}_m$ be a Poincaré line bundle on $\Pic^m(X)\times X$ normalized with respect to $x_0\in X$. Denote by $\pi$, $p$ the projections 
$$\pi:\Pic^m(X)\times X\to \Pic^m(X)\ ,\ p:\Pic^m(X)\times X\to  X\ .$$
 When the pair $(m,{\cal E}_0)$ is acyclic ${\cal Q}uot^{m}_{{\cal E}_0}$ can be identified with a projective bundle over $\Pic^m(X)$. The proof is based on the universal property of a Quot space. 
\begin{pr} \label{Quot-acyclic} Suppose the pair $(m,{\cal E}_0)$ is acyclic.
\begin{enumerate}[(i)]
\item The sheaf $\mathscr{V}:=R^0\pi_*(\mathscr{L}_m^\vee \otimes p^*({\cal E}_0))$ on $\Pic^m(X)$ is locally free.
\item $\ch(\mathscr{V})=\pi_*\big(\ch(\mathscr{L}_m^\vee \otimes p^*({\cal E}_0))\cup p^*(\td(X))\big)$.
\item Let $\nu:\P(\mathscr{V})\to\Pic^m(X)$, $p_X:\P(\mathscr{V})\times X\to X$, $\rho:\P(\mathscr{V})\times X\to  \P(\mathscr{V})$ be the natural projections, and let 
$$\mathscr{E}:=(\nu\times\id)^*(\mathscr{L}_m)\otimes  \rho^*({\cal O}_{\mathscr{V}}(-1))\ .$$
There exists a tautological monomorphism $j:\mathscr{E}\to p_X^*({\cal E}_0)$ such that the corresponding quotient $\mathscr{Q}$ of $p_X^*({\cal E}_0)$ is flat over $\P(\mathscr{V})$.
\item The obtained epimorphism  $u:p_X^*({\cal E}_0)\to \mathscr{Q}$ defines an isomorphism 
$$\P(\mathscr{V})\textmap{\simeq}{\cal Q}uot^m_{{\cal E}_0}\ .$$

\end{enumerate}

\end{pr}

\begin{proof}
Since the pair $(m,{\cal E}_0)$ is acyclic we have $h^0({\cal L}^\vee \otimes {\cal E}_0)=\chi({\cal L} ^\vee\otimes {\cal E}_0)$, which is a topological invariant of the pair $(m,{\cal E}_0)$, hence it is constant with respect to $[{\cal L}]\in\Pic^m(X)$.  The first statement follows from Grauert's locally freeness theorem. The second statement is a consequence of the Grothendieck-Riemann-Roch theorem for families taking into account that $R^i\pi_*(\mathscr{L}_m^\vee \otimes p^*({\cal E}_0))=0$ for $i>0$. For the third statement  note that the evaluation morphism
$$\mathrm{ev}: \mathscr{L}_m  \otimes \pi^*(\mathscr{V})\to p^*({\cal E}_0)
$$
on $\Pic^m(X)\times X$ can be regarded as an element of 
$$H^0(\mathscr{L}_m^\vee\otimes \pi^*(\mathscr{V})^\vee\otimes p^*({\cal E}_0))=H^0\left( (\nu\times\id_X)_* \big((\nu\times\id_X)^* (\mathscr{L}_m^\vee) \otimes{\cal O}_{\mathscr{V}}(1)\otimes p_X^*({\cal E}_0)\big)\right)  $$
$$ =H^0\big({\cal H}om\left(\big(\nu\times\id_X)^*(\mathscr{L}_m)\otimes{\cal O}_{\mathscr{V}}(-1), p_X^*({\cal E}_0)\big)\right),\ $$
hence it defines a morphism $(\nu\times\id_X)^*(\mathscr{L}_m)\otimes{\cal O}_{\mathscr{V}}(-1)\to  p_X^*({\cal E}_0)$ of sheaves on  $\P(\mathscr{V})\times X$. Using Remark \ref{smooth-fibres}, Lemma \ref{flat} it follows that the corresponding morphism 
$$ j:(\nu\times\id_X)^*(\mathscr{L}_m)\otimes{\cal O}_{\mathscr{V}}(-1)\to p_X^*({\cal E}_0)$$
is a sheaf monomorphism and the corresponding quotient $\mathscr{Q}$ is flat over $\P(\mathscr{V})$.

The fourth statement is a consequence  of the third. It suffices to prove that the quotient epimorphism $p_X^*({\cal E}_0)\to \mathscr{Q}$   satisfies the universal property of the tautological quotient over ${\cal Q}uot^m_{{\cal E}_0}$.

Let $Y$ be an arbitrary complex space and $q_X:Y\times X\to X$, $q_Y:Y\times X\to Y$ the projections on the two factors. Let $v:q_X^*({\cal E}_0)\to {\cal Q}$ be an epimorphism of coherent sheaves on $Y\times X$, such that ${\cal Q}$ is flat over $Y$ and the kernel ${\cal F}:=\ker(v)$ is fiberwise locally free of rank 1 and with Chern class $m$. Since ${\cal Q}$ is flat over $Y$ it follows that ${\cal F}$ is flat over $Y$ as well. Therefore ${\cal F}$ is a line bundle over $Y\times X$ and, using the universal property of the Picard group, there exists a morphism  $\alpha:Y\to\Pic^m(X)$, a line bundle ${\cal M}$ on $Y$ and a line bundle isomorphism
$$\beta:{\cal F}\to q_Y^*({\cal M})\otimes(\alpha\times\id_X)^*(\mathscr{L}_m)
$$
over $Y\times X$.  Since the base change property holds for $R^0\pi_*(\mathscr{L}_m^\vee \otimes p^*({\cal E}_0))$, we obtain   isomorphisms
$$\alpha^*(\mathscr{V})\simeq(q_Y)_*((\alpha\times\id)^*(\mathscr{L}_m^\vee\otimes p^*({\cal E}_0))\simeq(q_Y)_*((\alpha\times\id)^*(\mathscr{L}_m)^\vee\otimes q^*_X({\cal E}_0))\ .
$$
On the other hand, the monomorphism $\iota:{\cal F}\hookrightarrow q^*_X({\cal E}_0)$ defines a section $\sigma$ in ${\cal F}^\vee\otimes q^*_X({\cal E}_0)$ which is fiberwise non-trivial  by Lemma \ref{flat}. Using the isomorphism $\beta$ we obtain a fiberwise non-trivial section 
$$s\in H^0\left(q_Y^*({\cal M})^\vee\otimes(\alpha\times\id)^*(\mathscr{L}_m)^\vee\otimes q^*_X({\cal E}_0)\right)\ ,$$
 hence a nowhere vanishing  section $(q_Y)_*(s)$ in
 $$  (q_Y)_*(q_Y^*({\cal M})^\vee\otimes(\alpha\times\id)^*(\mathscr{L}_m)^\vee\otimes q^*_X({\cal E}_0))={\cal M}^\vee \otimes (q_Y)_* \left((\alpha\times\id)^*(\mathscr{L}_m)^\vee\otimes q^*_X({\cal E}_0) \right)$$
 $$= {\cal M}^\vee\otimes\alpha^*(\mathscr{V})\ .$$
   Therefore  $(q_Y)_*(s)$ identifies ${\cal M}$ with a line subbundle of $\alpha^*(\mathscr{V})$. This  line subbundle induces a morphism $\lambda:Y\to\P(\mathscr{V})$ with $\nu\circ\lambda=\alpha$. It is now easy  to see that the epimorphism $v$ can be identified with the pull-back of $u$ via $\lambda$, and that $\lambda$ is the unique morphism $Y\to \P(\mathscr{V})$ which induces such an identification.

\end{proof}

Let again $(h_i)$ be a basis of $H_1(X,\Z)/\Tors$, and let $\delta:H_1(X,\Z)/\Tors\to H^1(\Pic^m(X),\Z)$ be the natural isomorphism.
Put $\lambda_i:=\delta(h_i)\in H^1(\Pic^m(X),\Z)$.
Then one has $\delta=\sum_i \lambda_i\otimes h^i $. Now put $h^I:=h^{i_1}\cup\dots\cup h^{i_k}$, $\lambda_I:=\lambda_{i_1}\cup\dots\cup \lambda_{i_k}\in H^k(\Pic^m(X),\Z)$. Then a simple computation shows that 
$$ \delta^k=(-1)^{\frac{k(k-1)}{2}} \sum_{\substack{I\subset\{1,\dots b\} \\  |I|=k}}\lambda_I \otimes h^I\ .
$$

Therefore, for any class $\cg\in  H^{2n-k}(X,\Z)$ one can write
$$\pi_*(\delta^{k}\cup p^*(\cg))=(-1)^{\frac{k(k-1)}{2}}\pi_*\left\{ \sum_{\substack{I\subset\{1,\dots b\} \\  |I|=k}}\lambda_I \otimes (h^I\cup\cg)\right\}$$
$$=(-1)^{\frac{k(k-1)}{2}}
\sum_{\substack{I\subset\{1,\dots b\} \\  |I|=k}} \langle h^I\cup\cg,[X]\rangle \lambda_I\ .
$$
\def\Alt{\mathrm{Alt}}
In particular
$$\pi_*(\delta^2\cup p^*[\omega_g^{n-1}])=-\sum_{i,j}\langle h^i\cup h^j\cup[\omega_g^{n-1}],[X]\rangle=-2\theta\ .
$$
Identifying $H^*(\Pic^m(X),\Z)=\Alt^*(H^1(X,\Z),\Z)$  we have
$$\lambda_I(h^{j_1},\dots, h^{j_k})=\epsilon^J_I\ ,
$$
$$\pi_*(\delta^{k}\cup p^*(\cg))(h^{j_1},\dots, h^{j_k})=(-1)^{\frac{k(k-1)}{2}}
\sum_{\substack{I\subset\{1,\dots b\} \\  |I|=k}} \langle h^I\cup\cg,[X]\rangle \epsilon^J_I\ .
$$
Note that  
$$\sum_{\substack{I\subset\{1,\dots b\} \\  |I|=k}} h^I \epsilon^J_I = k!\  h^J\ .
$$
This proves:
\begin{re}\label{RR}
 For a class $\cg\in H^{2n-k}(X,\Z)$ define $\kg_\cg\in \Alt^k(H^1(X,\Z),\Z)$ by
$$\kg_\cg(x^1,\dots x^k):=\langle x_1\cup\dots \cup x_k\cup\cg,[X]\rangle\ .
$$
Using the canonical identification $\Alt^k(H^1(X,\Z),\Z)=H^k(\Pic^m(X),\Z)$ one has
$$\pi_*(\delta^{k}\cup p^*(\cg))=(-1)^{\frac{k(k-1)}{2}} k! \ \kg_\cg\ .
$$
\end{re}

Let $\omega_t$ be the Kähler form induced on ${\cal Q}uot^{m}_{{\cal E}_0}$ via the Kobayashi-Hitchin correspondence $KH$. Using Remark \ref{universalK} and formula (\ref{Perutz}) we obtain a formula for the  cohomology class $[\omega_t]$ in terms of the classes $\gamma$ and $\theta$ defined in section \ref{KF}.   

In our case, by Proposition \ref{Quot-acyclic} (iii) we get $\gamma=c_1({\cal O}_{\mathscr{V}}(1))$, hence
$$
[\omega_{t}]=4\pi^2\big(\nu^*(\theta)+\mg\gamma\big)+ 2t\pi(n-1)! \mathrm{Vol}_g(X)\gamma$$
$$
=4\pi^2\nu^*(\theta)+\big(4\pi^2\mg+2t\pi(n-1)! \mathrm{Vol}_g(X) \big)\gamma\ ,
$$
with
$$  \mg:=\langle m\cup[\omega_g]^{n-1},[X]\rangle=\deg_g(E)\ ,\  \theta=\sum_{i<j} h_{ij} \delta(h_i)\cup  \delta(h_i)\in H^2(\Pic^m(X))\ , $$
$$h_{ij}:=\langle h^i\cup h^j\cup [\omega^{n-1}],[X]\rangle\     .
$$

 Put 
 $$A:=4\pi^2 ,\ B_t:=4\pi^2\mg+2t\pi(n-1)! \mathrm{Vol}_g(X) ,\ R:=\rk(\mathscr{V})  ,\ q:=\frac{b_1(X)}{2} ,\ N:=q+R-1\ .
 $$
The volume $V_t$ of the Kähler manifold $({\cal Q}uot^{m}_{{\cal E}_0},\omega_t)$ is
$$V_t=\frac{1}{N!}\int_{\P(\mathscr{V})}[A\nu^*(\theta) +B_t \gamma]^{N}\ .
$$
 Projecting onto $\Pic^m(X)$ and using  formula (4.3) of \cite{ACGH} we get
\begin{equation}\label{V} V_t=\frac{1}{N!}\sum_{ R-1\leq k\leq N}  \left(\begin{matrix}N\\ k\end{matrix}\right) A^k B_t^{N-k} \bigg\langle   \theta^k \cup \nu_*(\gamma^{N-k}),[\Pic^m(X)]\bigg\rangle
$$
$$=\frac{1}{N!}\sum_{ R-1\leq k\leq N}  \left(\begin{matrix}N\\ k\end{matrix}\right) A^k B^{N-k}_t \bigg\langle   \theta^k \cup s_{q-k}(\mathscr{V}),[\Pic^m(X)]\bigg\rangle\ ,
\end{equation}
where $s_j(\mathscr{V})$ denotes the $j$-th Segre class of $\mathscr{V}$. The classes $\ch_i(\mathscr{V})$ can be computed using Proposition \ref{Quot-acyclic} (ii) and the result is
$$\ch_i(\mathscr{V})=\pi_*\left ( \frac{1}  {(2i)!} \delta^{2i}\cup  p^*\left(e^{-m}\cup\ch({\cal E}_0)\td(X) \right)^{(n-i)}  \right )\ .
$$

Now decompose 
$$\ch({\cal E}_0)\td(X)=\sum C_i 
$$
with $ C_i\in H^{2i}(X,\Q)$.  Then, using Remark \ref{RR} we obtain
$$\ch_i(\mathscr{V})=\frac{1}  {(2i)!} \sum_{s=0}^{n-i} \frac{(-1)^s}{s!}  \pi_*\left (  \delta^{2i} \cup p^*\left( m^s C_{n-i-s} \right)\right)= \sum_{s=0}^{n-i} \frac{(-1)^{i+s}}{s!} \kg_{m^s C_{n-i-s}}\ .
$$
 
Using formula (1.2) p. 156 in \cite{ACGH} we see that the Chern polynomial of $\mathscr{V}$ is given by 
$$1+\sum c_i(\mathscr{V}) =\exp\left({\sum_{i=1}^\infty \frac{(-1)^{i+1}}{i} \ch_i(\mathscr{V})  }\right)\ ,
$$
hence the  Segre polynomial is  
$$1+\sum s_i(\mathscr{V}) =\exp\left({\sum_{i=1}^\infty \frac{(-1)^{i}}{i} \ch_i(\mathscr{V})  }\right)=\exp\left({\sum_{i=1}^q \sum_{s=0}^{n-i} \frac{(-1)^s}{i s!} \kg_{m^s C_{n-i-s}}  }\right)\ .
$$

Suppose now that the basis $(h_i)$ of $H_1(X,\Z)/\Tors$ was chosen such that the image in $H^1(X,{\cal O})$ of the dual basis $(h^i)$ is compatible with the complex orientation.   With this choice the linear form $\langle\cdot ,[\Pic^m(X)]\rangle$ on $H^{2q}(\Pic^m(X),\Z)$ corresponds to the linear form
$$\phi\to\phi (h^1,\dots,h^{2q}) 
$$ 
on $\Alt^{2q}(H^1(X,\Z),\Z)$.  Using (\ref{V}) and noting $A^k B^{N-k}_t=(4\pi^2)^N(\deg(E)+\tg)^{N-k}$ we obtain
$$V_t=\frac{(4\pi^2)^N}{N!}\bigg\{\bigg(\sum_{k=R-1}^N  \left(\begin{matrix}N\\ k\end{matrix}\right)  (\deg_g(E)+\tg)^{N-k}    \theta^k \bigg)\wedge \ \ \ \ \ \ \ \ \ \ \ \ \ \ \ \ \ \ \ \ \ \ \ \ \ \ \ \ \ \ \ \ \ \ \ \ \ \ \ $$
$$\ \ \ \ \ \ \ \ \ \ \ \ \ \ \ \ \ \ \ \ \ \ \ \ \ \ \ \ \ \ \ \ \ \ \ \ \ \ \ \ \ \  \wedge \exp\bigg({\sum_{i=1}^q \sum_{s=0}^{n-i} \frac{(-1)^s}{i s!} \kg_{m^s C_{n-i-s}}  }\bigg)\bigg\}(h_1,\dots,h_{2q})\ .
$$

This shows that $V_t$ is a polynomial function in $t$ whose coefficients are determined by the cohomology classes $[\omega_g]$, $m$, $\ch_i({\cal E}_0)$, and $\td_i(X)$.







\subsection{Volumina of moduli spaces of twisted matrix divisors}

In this section we will compute the volume of a Quot space ${\cal Q}uot^E_{{\cal E}_0}$   in the special case when $r=r_0$  and the base $X$ is a complex curve of genus $\g$.   Since $X$ is a curve, ${\cal C}^\infty$ bundles of rank $r_0$ on $X$ are classified up to isomorphism by their degree, hence  the Quot space ${\cal Q}uot^E_{{\cal E}_0}$ is the space of coherent subsheaves ${\cal E}\hookrightarrow {\cal E}_0$ with $\rk({\cal E})=r_0$ and $\deg({\cal E})=\deg(E)$.  An element $({\cal E}\hookrightarrow {\cal E}_0)\in {\cal Q}uot^E_{{\cal E}_0}$ defines a sheaf monomorphism $\wedge^{r}({\cal E})\hookrightarrow \wedge^r({\cal E}_0)$, hence ${\cal Q}uot^E_{{\cal E}_0}=\emptyset$ for $\deg({\cal E}_0)<\deg(E)$ and ${\cal Q}uot^E_{{\cal E}_0}$ reduces to the singleton $\{\id_{{\cal E}_0}\}$ when $\deg({\cal E}_0)=\deg(E)$. We put $d:=\deg({\cal E}_0)-\deg(E)$ and we suppose that $d\geq 0$.
\vspace{3mm}
Recall the following well known result:

\begin{pr} \label{tangent}  Let $X$ be a complex curve, ${\cal E}_0$ (respectively $E$) a holomorphic (respectively differentiable) vector bundle of rank $r_0$  on $X$. Denote by  $\mathscr{E}$ and $\mathscr{Q}$  the universal kernel and the universal quotient  on ${\cal Q}uot^E_{{\cal E}_0}\times X$. The space ${\cal Q}uot^E_{{\cal E}_0}$ is   smooth of dimension $ rd$, and there is a canonical identification
\begin{equation}\label{tangentIso}{\cal T}_{{\cal Q}uot^E_{{\cal E}_0}}\textmap{\simeq} \rho_*(\mathscr{E}^\vee\otimes \mathscr{Q})\ ,
\end{equation}
where $\rho:{\cal Q}uot^E_{{\cal E}_0}\times X\to {\cal Q}uot^E_{{\cal E}_0}$ denotes the projection onto the first factor.
 \end{pr}

\subsubsection{The Abelian case}\label{AbCase} With the notations and under the conditions above suppose that $r=r_0=1$. In this case the Quot space ${\cal Q}uot^E_{{\cal E}_0}$ can be identified with the symmetric power  $X^{(d)}$ and, via this identification,  the universal kernel $\mathscr{E}$ and the universal quotient $\mathscr{Q}$  on ${\cal Q}uot^E_{{\cal E}_0}\times X$ can be identified with ${\cal O}(-\Delta)\otimes p_X^*({\cal E}_0)$ and $p_X^*({\cal E}_0)\otimes {\cal O}_\Delta$ respectively. Here $\Delta$ is the tautological divisor of $X^{(d)}\times X$. By Proposition \ref{tangent} we obtain  a canonical isomorphism
\begin{equation}\label{TgAB}
{\cal T}_{{\cal Q}uot^E_{{\cal E}_0}}\textmap{\simeq} \rho_*({\cal O}(\Delta)_\Delta )\ .
\end{equation}

By formula (\ref{cAbelian}) the Künneth decomposition  of $c_1 (\mathscr{E})$ reads %
\begin{equation}\label{cE}
c_1(\mathscr{E})=-\gamma\otimes 1+\delta+\deg(E)\otimes\{X\}\ ,
\end{equation}
where $\{X\}\in H^2(X,\Z)$ stands for the cohomological fundamental class of $X$ and we  denote by the same symbol the pull-back of $\delta$ via the map 
$$\nu\times\id_X:{\cal Q}uot^E_{{\cal E}_0}\times X\to \Pic^{\deg(E)}(X)\times X\ .$$  
Furthermore,   formula (\ref{c2Abelian}) gives
\begin{equation}\label{c2E}
c_1^2(\mathscr{E})/[X]=-2\theta - 2\deg(E) \gamma\ ,
\end{equation}
where we use the same symbol for the pull-back of  $\theta\in H^2(\Pic^{\deg(E)}(X),\Z)$ via the morphism $\nu:{\cal Q}uot^E_{{\cal E}_0}\to \Pic^{\deg(E)}(X)$.

The numbers obtained by evaluating the classes $ \gamma^{d-j} \theta^j$ on the fundamental class of ${\cal Q}uot^E_{{\cal E}_0}$ can be computed explicitly, and the result is:

\begin{equation}\label{numbers} \langle \gamma^{d-j} \theta^j , [{\cal Q}uot^E_{{\cal E}_0}]\rangle =\left\{
 \begin{array}{ccc}
 \frac{\g!}{(\g-j)!}  &\rm for& 0\leq j\leq \g\\
 0  &\rm for& j > \g
 \end{array}
 \right.
\end{equation}
This is the classical Poincaré formula (see for instance \cite{ACGH} p. 343), but can also  be obtained as a special case of Theorem 3.8 in \cite{OT2} by interpreting these numbers as gauge theoretical Gromov-Witten invariants.
Using this formula  and our formula (\ref{Perutz-new}) for the Kähler class $[\omega_t]$ on Abelian moduli spaces we obtain for the volume of ${\cal Q}uot^E_{{\cal E}_0}$ the formula:
$$V_t({\cal Q}uot^E_{{\cal E}_0})=(4\pi^2)^d\sum_{i=1}^{\min(d,\g)}\left(\begin{matrix} \g\\ i\end{matrix}\right)\frac{1}{(d-i)!}(\deg(E)+\tg)^{d-i}
$$
Up to the normalization factor $\pi^2$ it  specializes to the Manton-Nasir formula \cite{MN}  for the volume of $X^{(d)}$ when ${\cal E}_0={\cal O}_X$ and $t=\frac{1}{2}$.
\\

Let now $E_i$, ${\cal L}_i$  ($i\in\{1,2\}$) be differentiable, respectively holomorphic line bundles on $X$, and let   ${\cal Q}_i:={\cal Q}uot^{E_i}_{{\cal L}_i}$, $\mathscr{E}_i$ and  $\mathscr{Q}_i$ be the associated, Quot spaces, universal kernels and universal quotients. Put $l_i:=\deg({\cal L}_i)$, $m_i:=\deg(E_i)$, $d_i:=\deg({\cal L}_i)-\deg(E_i)$.

Consider the product ${\cal Q}_{12}:={\cal Q}_1\times {\cal Q}_2$ and denote by $\rho_{12}:{\cal Q}_{12}\times X\to {\cal Q}_{12}$, $p_i:{\cal Q}_{12}\to {\cal Q}_i$, $q_i:=p_i\times\id_X:{\cal Q}_{12}\times X\to {\cal Q}_{i}\times X$ the natural projections. For later computations we will need the characteristic classes  of the  push-forward:
$${\cal N}_{12}:=(\rho_{12})_*\left\{q_1^*(\mathscr{E}_1)^\vee\otimes q_2^*(\mathscr{Q}_2)\right\}\ ,
$$
which is a locally free sheaf of rank $d_2$ on ${\cal Q}_{12}$. Since  the higher direct images of the right hand sheaf on ${\cal Q}_{12}$ vanish, we obtain by the Grothendieck-Riemann-Roch theorem
\begin{equation}\label{chN12}
\ch({\cal N}_{12})=(-\bar\g+l_2-m_1-\theta_1 ) e^{\gamma_1}+ (\bar \g+m_1-m_2  +(\theta_1+\theta_2+\sigma_{12})  )e^{\gamma_1-\gamma_2}\ ,
\end{equation}
where   $\theta_i$, $\gamma_i$ stand  for the pull-backs of the corresponding classes on ${\cal Q}_i$ and 
$$\sigma_{12}:=(\rho_{12})_*(q_1^*(\delta_1)\cup q_2^*(\delta_2))\ .$$
 Since ${\cal Q}_{12}$ has torsion-free integral cohomology,
formula (\ref{chN12}) implies
\begin{equation}\label{ctN12}
c_t({\cal N}_{12})=\frac{(1+t\gamma_1)^{-\bar\g+ l_2-m_1} \ e^{-\frac{t\theta_1}{1+t\gamma_1}}}{(1+t(\gamma_1-\gamma_2))^{-\bar \g- m_1+m_2}\ e^{-\frac{t(\theta_1+\theta_2+\sigma_{12})}{1+t(\gamma_1-\gamma_2)}}}\ .
\end{equation}

 \subsubsection{The Grothendieck embedding}
 
 We begin with a simple result which shows that the Kähler classes associated with any Grothendieck embedding of ${\cal Q}uot^E_{{\cal E}_0}$ in a projective space  coincides (up to a universal factor) with the Kähler class  induced  from the moduli space  ${\cal M}_t^*(E,E_A,A_0)$ via the Kobayashi-Hitchin correspondence   for a suitable choice of the parameter $t$. Therefore our computation will also give the degree of this Quot space with respect to the corresponding Grothendieck embedding.

 Fix $x_0\in X$, and let $n\in\N$  be sufficiently large so that $h^1({\cal E}_0(nx_0))=h^1({\cal E}(nx_0))=0$ for every ${\cal E}\subset {\cal E}_0$ with $\deg({\cal E})=m$. Put 
 $$V:=H^0({\cal E}_0(nx_0))\ ,\ s:= \deg(E)+r_0(n-\g+1) \ . $$
 We obtain a holomorphic map $\iota_n:{\cal Q}uot^E_{{\cal E}_0}\to \G r_s(V)$ given by
 $$\iota_n({\cal E}\subset{\cal E}_0):=H^0({\cal E}(nx_0))\subset V \ .
 $$  

The map $\iota_n$  is an embedding for sufficiently large $n\in\N$, hence composing with the Plücker embedding $Pl:\G r_s(V)\to\P(\wedge^s(V))$ we obtain a projective embedding 
$$j_n:=Pl\circ \iota_n: {\cal Q}uot^E_{{\cal E}_0}\to \P(\wedge^s(V))\ .$$
   Denote by ${\cal U}$ the tautological subbundle of $\G r_s(V)$. We have   obvious isomorphisms
$$Pl^*({\cal O}_{ \wedge^s(V)}(1))\simeq \det({\cal U})^\vee\ ,\ 
\iota_n^*({\cal U})\simeq \rho_*(\mathscr{E}(n x_0))\ .
$$
Here we used the notation $\mathscr{E}(n x_0):= \mathscr{E}\otimes p_X^*({\cal O}(n x_0))$. Therefore  
$$c_1(j_n^*({\cal O}_{\wedge^s(V)}(1)))=-c_1(\rho_*(\mathscr{E} (n x_0))\ .
$$

Denote by $\{X\}\in H^2(X,\Z)$ the cohomological fundamental class of $X$ and put $\bar \g:=\g-1$. Since $R^1\rho_*({\cal E} (n x_0))=0$, we obtain using the Grothendieck-Riemann-Roch theorem:
$$c_1(j_n^*({\cal O}_{\wedge^s(V)}(1)))=-c_1(\rho_!(\mathscr{E} (n x_0))=$$
$$-\rho_*\left\{\mathrm{ch}(\mathscr{E})p_X^*(e^{n\{X\}}\cup (1-\bar\g\{X\} ) \right\}^{(2)}=
-\rho_*\left\{\mathrm{ch}(\mathscr{E})p_X^* (1+(n-\bar \g)\{X\} ) \right\}^{(2)}
$$
$$= \rho_*\left(-\mathrm{ch}_2(\mathscr{E})-(n-\bar \g)c_1(\mathscr{E}) \cup p_X^*(\{X\})\right)
$$

On the other hand, by formula (\ref{Kformvortex})   the class of the Kähler metric $\omega_t$ on ${\cal Q}uot^E_{{\cal E}_0}$ induced from the moduli space  ${\cal M}_t(E,E_A,A_0)$ via the Kobayashi-Hitchin correspondence (see Remark \ref{universalK})  is 
 
$$ [\omega_t]=4\pi^2\rho_*\left[ -\mathrm{ch}_2(\mathscr{E}) -\frac{t\Vol_g(X)}{2\pi}    c_1(\mathscr{E}) \wedge   \{X\}\right] \ .
$$ 
Here $g$ is  a Kähler metric  on $X$.
This proves:
\begin{pr}\label{Gr}
The Kähler class $c_1(j_n^*({\cal O}_{ \wedge^s(V)}(1)))$ of the Grothendieck embedding $j_n: {\cal Q}uot^E_{{\cal E}_0}\to \P(\wedge^s(V))$ compares to the Kähler class $[\omega_t]$ induced from the moduli space  ${\cal M}_t(E,E_A,A_0)$ via the Kobayashi-Hitchin correspondence as follows:
\begin{equation}\label{Gr-formula}
c_1(j_n^*({\cal O}_{ \wedge^s(V)}(1)))=\frac{1}{4\pi^2} \left[\omega_{\frac{2\pi(n-\bar \g)}{\Vol_g(X)} }\right]\ .
\end{equation}

\end{pr}

\subsubsection{Localization}

Suppose that ${\cal E}_0$ decomposes as
$${\cal E}_0=\oplus_{i=1}^r {\cal L}_i\ ,
$$
where ${\cal L}_i$ are line bundles on $X$. Putting $l_i:=\deg({\cal L}_i)$ we have 
$$l:=\sum_{i=1}^r l_i=\deg({\cal E}_0) \ .$$

Let  
$$I(d):=\{\underline{d}=(d_1,\dots,d_r)\in \N^r|\   \sum_{i=1}^r d_i=d\} $$
be   the set of weak length $r$ decompositions of $d$. For every $\underline{d}\in I(d)$ define
$$Q^{\underline{d}}:=\prod_{i=1}^r {\cal Q}uot^{l_i-d_i}_{{\cal L}_i} \ .
$$
Note that  we have an obvious embedding
$$j_{\underline{d}}: Q^{\underline{d}}\to {\cal Q}uot^E_{{\cal E}_0}
$$
defined by
$$j_{\underline{d}}({\cal E}_1,\dots {\cal E}_r):=\bigoplus_{i=1}^r {\cal E}_i
$$
for every system $({\cal E}_1,\dots {\cal E}_r)$ of  rank 1  subsheaves  ${\cal E}_i\subset {\cal L}_i$ with $\deg({\cal E}_i)= l_i-d_i$.
\begin{re}\label{sum}
For every ${\underline{d}}\in I(d)$ there is a canonical isomorphism
\begin{equation}\label{DecSum}\resto{\mathscr{E}}{Q^{\underline{d}}}\simeq \bigoplus_{i=1}^r (p_i^{{\underline{d}}}\times\id_X)^*(\mathscr{E}^{\underline{d}}_i)\ ,
\end{equation}
where $p_i^{\underline{d}}: Q^{\underline{d}}\to {\cal Q}uot^{l_i-d_i}_{{\cal L}_i}$  denotes the projection onto the $i$-th factor, and $\mathscr{E}^{\underline{d}}_i$ stands for the universal kernel on ${\cal Q}uot^{l_i-d_i}_{{\cal L}_i}\times X$.

\end{re}

Endow  now the moduli space ${\cal Q}uot^E_{{\cal E}_0}$ with the $\C^*$-action associated with the morphism $\C^*\to\Aut({\cal E}_0)$ given by $z\mapsto \bigoplus_{i=1}^r z^{w_i}\id_{{\cal L}_i}$ (see Remark \ref{Symm}). Note that for every $\underline{d}\in I(d)$ one has ${\cal Q}^{\underline{d}}\subset [{\cal Q}uot^E_{{\cal E}_0}]^{\C^*}$  and, endowing the universal kernel $\mathscr{E}^{\underline{d}}_i$ with the $\C^*$-action $z\mapsto z^{w_i}\id_{\mathscr{E}^{\underline{d}}_i}$,  the isomorphism (\ref{DecSum}) becomes an isomorphism of $\C^*$-bundles over the trivial $\C^*$-space ${\cal Q}^{\underline{d}}$.  We will need the equivariant  first Chern class of the right hand summands.
Via the standard isomorphism $H^*_{\C^*}({\cal Q}^{\underline{d}},\R)\simeq H^*({\cal Q}^{\underline{d}},\R)[u]$ we get by (\ref{cE}), (\ref{c2E}) formulae of the form:
\begin{equation}\label{cEeq}
c_1^{\C^*}(\mathscr{E}^{\underline{d}}_i)=-(\gamma^{\underline{d}}_i-w_iu)\otimes 1+ \delta^{\underline{d}}_i+(l_i-d_i)\otimes\{X\}\ ,
\end{equation}
\begin{equation}\label{c2Eeq}
c_1^{\C^*}(\mathscr{E}^{\underline{d}}_i)^2/[X]=-2\theta^{\underline{d}}_i - 2(l_i-d_i) (\gamma^{\underline{d}}_i-w_iu)\ ,
\end{equation}
where $\gamma^{\underline{d}}_i$, $\delta^{\underline{d}}_i$, $\Theta^{\underline{d}}_i$ are the classes associated with the Quot space ${\cal Q}uot^{l_i-d_i}_{{\cal L}_i}$ as explained in section \ref{AbCase}. On the other hand, by Theorem \ref{Symm} and Remarks \ref {new-Rem}, \ref{universalK} we know that the class
$$\Omega_t:=\left[ -\mathrm{ch}_2^{\C^*}(\mathscr{E}) -\frac{t\Vol_g(X)}{2\pi}    c_1^{\C^*}(\mathscr{E}) \cup   \{X\}\right] /[X]
$$
is a lift of $\frac{1}{4\pi^2}[\omega_t]$ in $H^2_{\C^*}({\cal Q}uot^E_{{\cal E}_0},\R)$.

Using Remark \ref{sum} and the additivity of the equivariant Chern character with respect to direct sums we obtain
\begin{equation}\label{NiceSum}\resto{\Omega_t}{Q^{\underline{d}}}=\sum_{i=1}^r (p_i^{\underline{d}})^*\left\{\left[ -\frac{1}{2} c_1^{\C^*}(\mathscr{E}^{\underline{d}}_i)^2 -\frac{t\Vol_g(X)}{2\pi}   c_1^{\C^*}(\mathscr{E}^{\underline{d}}_i) \cup   \{X\}\right]/[X]\right\} $$
$$=\sum_{i=1}^r  [(s_i^{\underline{d}}\gamma^{\underline{d}}_i+\theta^{\underline{d}}_i)-s_i^{\underline{d}}w_i u]\ ,
\end{equation}
where
$$s_i^{\underline{d}}:=t\frac{\Vol_g(X)}{2\pi} +(l_i-d_i)\ ,
$$
and on the right we have omitted the symbol $(p_i^{\underline{d}})^*$ in front of $(s_i\gamma^{\underline{d}}_i+\theta^{\underline{d}}_i)$ to save on notations.

We will also need the equivariant Euler class of the normal bundle ${\cal N}_{Q^{\underline{d}}}$   of  $Q^{\underline{d}}$ in ${\cal Q}uot^E_{{\cal E}_0}$.  Using the isomorphism (\ref{tangentIso}) together with cohomology and base change we obtain the  direct sum decompositions
\begin{equation}\label{decN}\resto{{\cal T}_{{\cal Q}uot^E_{{\cal E}_0}}}{Q^{\underline{d}}}=\bigoplus_{i,j} (p^{\underline{d}}_{ij})^* {\cal N}^{\underline{d}}_{ij}\ ,\ {\cal N}_{Q^{\underline{d}}}=\bigoplus_{i\ne j} (p^{\underline{d}}_{ij})^*{\cal N}^{\underline{d}}_{ij}\ .
\end{equation}
Here ${\cal N}^{\underline{d}}_{ij}$ denotes the rank $d_j$ bundle on the product ${\cal Q}uot^{l_i-d_i}_{{\cal L}_i}\times {\cal Q}uot^{l_j-d_j}_{{\cal L}_j}$ studied in section \ref{AbCase}, and $p^{\underline{d}}_{ij}$ denotes the projection of $Q^{\underline{d}}$ onto this product. These decompositions are $\C^*$-invariant   and   $\C^*$ acts with weight $w_j-w_i$ on the summand ${\cal N}^{\underline{d}}_{ij}$. Using (\ref{ctN12}) we obtain
$$e^{\C^*}({\cal N}^{\underline{d}}_{ij})=((w_j-w_i) u)^{d_j} c_{\frac{1}{(w_j-w_i) u}}({\cal N}_{ij})
$$
$$=((w_j-w_i) u)^{d_j} \frac{\left[1+\frac{1}{(w_j-w_i) u}\gamma_i\right]^{-\bar \g-l_i+d_i+ l_j} \ e^{-\frac{\theta_i}{(w_j-w_i) u+\gamma_i}}}
{\left[1+\frac{1}{(w_j-w_i) u}(\gamma_i-\gamma_j)\right]^{-\bar \g -l_i+d_i+l_j-d_j}\ 
e^{-\frac{\theta_i+\theta_j+\sigma_{ij}}{(w_j-w_i) u+(\gamma_i-\gamma_j)}}} 
$$
$$=   \frac{\left((w_j-w_i) u+ \gamma_i\right)^{-\bar \g-l_i+d_i+ l_j} \ e^{-\frac{\theta_i}{(w_j-w_i) u+\gamma_i}}}
{\left((w_j-w_i) u+(\gamma_i-\gamma_j)\right)^{-\bar \g -l_i+d_i+l_j-d_j}\ 
e^{-\frac{\theta_i+\theta_j+\sigma_{ij}}{(w_j-w_i) u+(\gamma_i-\gamma_j)}}}\ .
$$
By (\ref{decN}) this gives
$$e^{\C^*}({\cal N}_{Q^{\underline{d}}})=\prod_{i\ne j} \frac{\left((w_j-w_i) u+ \gamma_i\right)^{-\bar \g-l_i+d_i+ l_j} \ e^{-\frac{\theta_i}{(w_j-w_i) u+\gamma_i}}}
{\left((w_j-w_i) u+(\gamma_i-\gamma_j)\right)^{-\bar \g -l_i+d_i+l_j-d_j}e^{-\frac{\theta_i+\theta_j+\sigma_{ij}}{(w_j-w_i) u+(\gamma_i-\gamma_j)}}}$$
$$=(-1)^{\bar\g\left(\begin{smallmatrix}  r\\   2\end{smallmatrix}\right)+(r-1)(l-d)}
\frac{\prod_{i\ne j}\left((w_j-w_i) u+ \gamma_i\right)^{-\bar \g-l_i+d_i+ l_j} \ e^{-\frac{\theta_i}{(w_j-w_i) u+\gamma_i}}}
{\prod_{i<j}\left((w_j-w_i) u+(\gamma_i-\gamma_j)\right)^{-2\bar \g}}\ .
$$

From now on we suppose that the weights $w_i$ are pairwise distinct.  This implies 
$$\{{\cal Q}uot^E_{{\cal E}_0}\}^{\C^*}=\coprod_{\underline{d}\in I(d)} Q^{\underline{d}}\ ,
$$
so that we can use the integration formula (\cite{AB} (3.8)) to compute the volume $V_t$ of  ${\cal Q}uot^E_{{\cal E}_0}$ with respect to the Kähler form $\frac{1}{4\pi^2} \omega_t$:
 $$V_t=\frac{1}{(rd)!}\sum_{\underline{d}\in I(d)} \left\langle \left\{  \frac{(\resto{[\Omega_t]}{Q^{\underline{d}}})^{rd}}{ e^{\C^*} ({\cal N}_{Q^{\underline{d}}} )} \right\},\left[Q^{\underline{d}}\right]\right\rangle$$
Here  the expression $  \frac{(\resto{\Omega_t}{Q^{\underline{d}}})^{rd}}{ e^{\C^*} ({\cal N}_{Q^{\underline{d}}} )}$ is regarded as an element in the ring $H^*(Q^{\underline{d}},\R)[[u,u^{-1}]]$. The degree with respect to $u$ of the terms with coefficients in $H^{2d}(Q^{\underline{d}},\R)$ is 0, hence the formula above yields a real number.  

Using our formulae for $[\Omega_t]$ and $e^{\C^*} ({\cal N}_{Q^{\underline{d}}})$ we obtain the following formula for the volume $V_t$:
 \\

\begin{equation}\label{explicit}
V_t=(-1)^{\bar\g\left(\begin{smallmatrix}  r\\   2\end{smallmatrix}\right)+(r-1)(l-d)}\frac{1}{(rd)!}\sum_{\underline{d}\in I(d)}$$
$$\left\langle
\frac{\left( \displaystyle{\sum_{i=1}^r}[(  s^{\underline{d}}_i\gamma^{\underline{d}}_i+\theta^{\underline{d}}_i)-s^{\underline{d}}_iw_i u]\right)^{rd}\displaystyle{\prod_{i\ne j}}\left((w_j-w_i) u+ \gamma_i^{\underline{d}}\right)^{\bar \g+l_i-d_i- l_j} \ e^{\frac{\theta_i^{\underline{d}}}{(w_j-w_i) u+\gamma_i^{\underline{d}}}}}
{\displaystyle{\prod_{i<j}}\left((w_j-w_i) u+(\gamma_i^{\underline{d}}-\gamma_j^{\underline{d}})\right)^{2\bar \g}},\left[Q^{\underline{d}}\right]\right\rangle,
\end{equation}
where
$$s^{\underline{d}}_i=t\frac{\Vol_g(X)}{2\pi}+l_i-d_i\ . 
$$

Formula (\ref{explicit}) provides an algorithm which generates explicit formulae for $V_t$: \begin{itemize}
\item For $\underline{d}\in I(d)$, $w=(w_1,\dots,w_r)\in\C^r$ (with $w_i\ne w_j$ for $i\ne j$) and $u\in\C^*$ put
$$F^{\underline{d}}_{w,u}(x_1,y_1,\dots,x_r,y_r):=$$
$$\frac{\left( \displaystyle{\sum_{i=1}^r}[(  s^{\underline{d}}_ix_i+y_i)-s^{\underline{d}}_iw_i u]\right)^{rd}\displaystyle{\prod_{i\ne j}}\left((w_j-w_i) u+ x_i\right)^{\bar \g+l_i-d_i- l_j} \ e^{\frac{y_i}{(w_j-w_i) u+x_i}}}
{\displaystyle{\prod_{i<j}}\left((w_j-w_i) u+(x_i-x_j)\right)^{2\bar \g}}\ ,
$$
and write explicitly the multi-homogeneous part 
$$\sum_{\alpha_i+\beta_i=d_i}[F^{\underline{d}}_w]_{\alpha_1,\beta_1\dots,\alpha_r,\beta_r}x_1^{\alpha_1}y_1^{\beta_1}\dots x_r^{\alpha_r}y_r^{\beta_r}
$$
of multi-degree $(d_1,\dots, d_r)$ in the Taylor expansion  of $F^{\underline{d}}_{w,u}$ at 0 . This expression has degree 0 with respect to $u$, hence we omitted $u$ in the notation of the coefficients.
\item Write
$$V_t=\frac{1}{(rd)!}\sum_{\underline{d}\in I(d)} \sum_{\alpha_i+\beta_i=d_i}[F^{\underline{d}}_{w}]_{\alpha_1,\beta_1\dots,\alpha_r,\beta_r}\prod_{i=1}^r \g(\g-1)\dots(\g-\beta_i+1)\ .$$
This final result is independent of the (pairwise distinct)  weights $w_i$. One can use the freedom to choose  the weights $w_i$ in order to simplify the computations (see for instance \cite{MO2}).
\end{itemize}

Using this algorithm to produce a closed formula  for arbitrary $r$, $d$ and $l$ seems to be  difficult.  We have computed the following examples:
\\
\\
{\bf Examples:} Let $r=2$ and put $\mu:=\mu({\cal E}_0)=\frac{l}{2}$, $\tg=\frac{t\Vol_g(X)}{2\pi}$.
\begin{enumerate}
\item $d=1$:
$$V_t=\tg+\mu+\bar\g\ .$$
\item $d=2$, $l$ even:
$$ V_t=\frac{1}{4!}\bigg\{ 4(\tg+\mu+\bar\g)\big[3(\tg+\mu+\bar\g)-4 \big]-6\bar\g\bigg\}\ .
 $$
\end{enumerate}

For $t_n= \frac{2\pi(n-\bar \g)}{\Vol_g(X)}$  (or, equivalently, $\tg_n:=n-\bar\g$) one gets the degree of the image of the Grothendieck embedding 
$j_n: {\cal Q}uot^E_{{\cal E}_0}\to \P(\wedge^s(V))$.
\begin{enumerate}
\item $d=1$:
$$\deg(j_n({\cal Q}uot^E_{{\cal E}_0}))=2n+l\ .
$$
\item   $d=2$ $l$ even:
$$\deg(j_n({\cal Q}uot^E_{{\cal E}_0}))= (2n+l)[3(2n+l)-8]-6\bar\g\ .
$$
\end{enumerate}

{\ }
\vspace{10mm}  \\
{\small Christian Okonek: \\
Institut f\"ur Mathematik, Universit\"at Z\"urich,
Winterthurerstrasse 190, CH-8057 Z\"urich,\\
e-mail: okonek@math.uzh.ch
\\  \\
Andrei Teleman: \\
CMI,   Aix-Marseille Universit\'e,  LATP, 39  Rue F. Joliot-Curie,  13453
Marseille Cedex 13,   e-mail: teleman@cmi.univ-mrs.fr
}

\setcounter{section}{-1}

\end{document}